\documentclass[11pt]{article}
\usepackage{amsmath,fullpage,amssymb,amsthm,hyperref,enumitem,bm}
\usepackage[dvipsnames]{xcolor}
\usepackage{tikz,ifthen,youngtab,mathtools}
\usetikzlibrary{matrix}

\newtheorem{theorem}{Theorem}[section]
\newtheorem{proposition}[theorem]{Proposition}
\newtheorem{lemma}[theorem]{Lemma}
\newtheorem{corollary}[theorem]{Corollary}
\newtheorem{conjecture}[theorem]{Conjecture}  
\newtheorem{problem}[theorem]{Problem}

\theoremstyle{definition}
  
\newtheorem{definition}[theorem]{Definition}   

\theoremstyle{remark}
\newtheorem{remark}[theorem]{Remark}  

\renewcommand{\S}{\mathcal{S}}
\DeclareMathOperator\des{des}
\DeclareMathOperator\Des{Des}
\DeclareMathOperator\asc{asc}
\DeclareMathOperator\Asc{Asc}
\DeclareMathOperator\hdes{hdes}
\DeclareMathOperator\HDes{HDes}
\DeclareMathOperator\maj{maj}
\DeclareMathOperator\SYT{SYT}
\newcommand{\card}[1]{{\lvert #1 \rvert}}
\newcommand{\D}{\mathcal{D}}
\newcommand\uu{\texttt{u}}
\newcommand\dd{\texttt{d}}
\newcommand\nn{\texttt{n}}
\newcommand\ee{\texttt{e}}
\newcommand\ua{{\bm\uparrow}}
\newcommand\da{{\bm\downarrow}}
\newcommand\uar{\textcolor{red}{\ua}}
\newcommand\dar{\textcolor{red}{\da}}
\newcommand\darr{\textcolor{teal}{\da}}
\newcommand\uarr{\textcolor{teal}{\ua}}
\newcommand\darrr{\textcolor{violet}{\da}}
\newcommand\uarrr{\textcolor{violet}{\ua}}
\newcommand\darrrr{\textcolor{cyan}{\da}}
\newcommand\uarrrr{\textcolor{cyan}{\ua}}
\newcommand{\Ld}{\mathcal{L}^\downarrow}
\newcommand{\Tu}{\mathcal{T}^\uparrow}
\newcommand{\Lu}{\mathcal{L}^\uparrow}
\newcommand{\Td}{\mathcal{T}^\downarrow}
\newcommand{\ol}{\overline}
\newcommand{\wh}{\widehat}
\newcommand{\cL}{\mathcal{L}}
\newcommand{\LK}{\mathrm{LK}}
\newcommand{\TD}{\Delta}
\newcommand{\bijd}{\varphi} 
\newcommand{\bija}{\psi} 
\newcommand\rev{\beta}	
\newcommand\swap{\leftrightarrow}
\renewcommand{\b}{b}
\newcommand\rot{\theta}	
\newcommand\es{\epsilon}
\DeclareMathOperator\evac{evac}
\newcommand\rowmotion{\rho} 
\DeclareMathOperator\st{st}
\newcommand{\C}{\mathcal{C}}
\newcommand{\row}{R_T}
\newcommand\omrow{\omega_{\mathrm{row}}}
\newcommand\omcol{\omega_{\mathrm{col}}}
\DeclareMathOperator\usualrowmotion{Row}
\DeclareMathOperator\AscCell{AscCell}
\DeclareMathOperator\HDesCell{HDesCell}

\newcommand{\drawtab}[3]{
\draw[Apricot,thin] (0,0) grid (#1,#2);
\foreach \ro [count=\i] in #3
	\foreach \num [count=\j] in \ro
		\node[scale=.85,blue] at (\j-.5,#2+.5-\i) {$\num$};
}
\newcommand{\drawarr}[2]{
\foreach \ro [count=\i] in #2
	\foreach \num [count=\j] in \ro
		\node at (\j-1,#1+.5-\i) {$\num$};
}

\def\e{-- ++(1,0)}
\def\n{-- ++(0,1)}
\newcommand\doublegrid[1]{
 \draw[thin,dotted] (0,0) grid (#1,#1);
 \draw (0,0) rectangle (#1,#1);
 \draw (0,0)--(#1,#1); 
}

\newcommand{\drawtabgen}[2]{
\foreach \ro [count=\i] in #2
	\foreach \num [count=\j] in \ro
		{\draw[Apricot,thin] (\j-1,#1-\i) rectangle (\j,#1-\i+1);
		\node[scale=.85,blue] at (\j-.5,#1+.5-\i) {$\num$};}
}

\definecolor{darkgreen}{rgb}{0,0.6,0}

\title{Symmetry of ascent and descent distributions\\ on rectangular and staircase tableaux}

\author{Sergi Elizalde\thanks{Department of Mathematics, Dartmouth College, Hanover, NH 03755. \texttt{sergi.elizalde@dartmouth.edu}}}

\date{}

\begin{document}

\maketitle

\begin{abstract}
We give direct bijective proofs of the symmetry of the distributions of the number of ascents and descents over standard Young tableaux of shape $\lambda$, where $\lambda$ is a rectangle $(n,n,\dots,n)$ or a truncated staircase $(n,n-1,\dots,n-k+1)$.
These can be viewed as instances of the more general symmetry of the distribution of descents over linear extensions of graded posets, 
for which previous proofs by Stanley and Farley were based on the theory of $P$-partitions and the involution principle, respectively.
In the case of two-row rectangles $(n,n)$, our bijection is equivalent to the Lalanne--Kreweras involution on Dyck paths, which bijectively proves the symmetry of the Narayana numbers. 

Our bijections are defined in terms of certain arrow encodings of standard Young tableaux. This setup allows us to construct other statistic-preserving involutions on tableaux of rectangular shape, providing 
a simple proof of the fact that  ascents and descents are equidistributed up to a shift, 
and proving a conjecture of Sulanke about certain statistics in the case of three rows. 
Finally, we use our bijections to define a possible notion of rowmotion on standard Young tableaux of rectangular shape, and to give a bijective proof of the symmetry of the number of descents on canon permutations, which have been recently studied as a variation of Stirling and quasi-Stirling permutations.
\end{abstract}



\section{Introduction}\label{sec:intro}

Let $\D_n$ be the set of Dyck paths of semilength $n$, defined as lattice paths from $(0,0)$ to $(2n,0)$ with steps $\uu=(1,1)$ and $\dd=(1,-1)$ that never go below the $x$-axis. A classical result in enumerative combinatorics is that the number of such paths with $h$ valleys (i.e., consecutive pairs $\dd\uu$), or equivalently $h+1$ peaks (i.e., consecutive pairs $\uu\dd$), is given by the Narayana\footnote{Despite being named after Narayana~\cite{venkata_narayana_sur_1955}, these numbers appear in earlier work of MacMahon~\cite{macmahon_combinatory_1960}.} numbers, which we denote by 
$$N(2,n,h)=\frac{1}{n}\binom{n}{h}\binom{n}{h+1}.$$

A consequence of this formula is that 
\begin{equation}\label{eq:symmetryN2} N(2,n,h)=N(2,n,n-h-1),\end{equation}
which can be interpreted as saying that the number of paths in $\D_n$ with $h+1$ peaks equals the number of those with $n-h$ peaks. A bijective proof of this non-obvious symmetry is given by an involution on $\D_n$ that was first considered by Kreweras~\cite{kreweras_sur_1970} and later studied by Lalanne~\cite{lalanne_involution_1992}, often referred to as the Lalanne--Kreweras involution~\cite{hopkins_birational_2022,elizalde_descents_2024}. We will give a description in Section~\ref{sec:LK}.

To generalize this symmetry, we will view Dyck paths as standard Young tableaux of rectangular shape $2\times n$. Let us first introduce some notation. Throughout the paper, let $\lambda=(\lambda_1,\lambda_2,\dots,\lambda_k)$ be a partition of $N$, that is, a weakly decreasing sequence of positive integers that add up to $N$, and let $\SYT(\lambda)$ denote the set of standard Young tableaux of shape $\lambda$. We draw tableaux in English notation, where the $r$th row from the top has $\lambda_r$ cells for $1\le r\le k$. The entries in each row increase from left to right, and entries in each column increase from top to bottom. 

For $1\le i\le N-1$, we say that $i$ is a descent (resp.\ ascent) of $T\in\SYT(\lambda)$ if $i$ lies in a row that is higher (resp.\ lower) than the row containing $i+1$. Denote by $\Des(T)$ and $\Asc(T)$ the sets of descents and ascents of $T$, respectively, and denote the cardinalities of these sets by $\des(T)=\card{\Des(T)}$ and $\asc(T)=\card{\Asc(T)}$. 
For example, if $T$ is the standard Young tableau on the left of Figure~\ref{fig:arrows}, then $\Des(T)=\{2,3,6,10,11,12,15,18,19,22,24,27,28\}$ and $\Asc(T)=\{4,5,7,8,14,16,20,21,23,25,26\}$, so $\des(T)=13$ and $\asc(T)=11$.

We are particularly interested in partitions of the form $\lambda=(n^k)$, which consist of $k$ parts equal to $n$. Then $\SYT(n^k)$ is the set of standard Young tableaux of rectangular shape with $k$ rows and $n$ columns.
In the rest of the paper, we assume that $k,n\ge1$.

For $k\ge2$, there is a straightforward bijection \begin{equation}\label{eq:TD} \TD:\SYT(n^2)\to\D_n, \end{equation} defined by letting the $i$th step of the path be $\uu$ if $i$ is in the first row of the tableau, and $\dd$ if it is in the second row, for each $1\le i\le 2n$. Under this bijection, ascents and descents of the tableau correspond to valleys and peaks of the Dyck path, respectively.
We can thus interpret equation~\eqref{eq:symmetryN2} as stating that the number of tableaux in $\SYT(n^2)$ with $h$ ascents (equivalently, $h+1$ descents) equals the number of those with $n-h-1$ ascents (equivalently, $n-h$ descents). 

In general, for standard Young tableaux of rectangular shape with an arbitrary number $k$ of rows, the distribution of the number of ascents is given by the {\em generalized Narayana} (or {\em $k$-Narayana}) {\em numbers}, which already appear in work of MacMahon from over a century ago~\cite{macmahon_combinatory_1960}. More recently, these numbers were studied by Sulanke~\cite{sulanke_generalizing_2004,sulanke_three_2005} in the context of higher-dimensional lattice paths. They can be defined as
\begin{equation}\label{eq:Nasc}  N(k,n,h)=|\{T\in\SYT(n^k):\asc(T)=h\}|.
\end{equation}
It is shown in~\cite{sulanke_generalizing_2004}, and we will prove again later, that they also have an equivalent definition in terms of descents as
\begin{equation}\label{eq:Ndes} 
N(k,n,h)=|\{T\in\SYT(n^k):\des(T)=h+k-1\}|.
\end{equation}

Sulanke obtained the following formula for the generalized Narayana numbers using Stanley's theory of $P$-partitions~\cite{stanley_enumerative_2012}, and he noted that it is implicit in MacMahon's work on plane partitions~\cite{macmahon_combinatory_1960}. Here we state it in a slightly different form.

\begin{theorem}[{\cite[Prop.~1]{sulanke_generalizing_2004}}]\label{thm:Narayana-formula}
For $0\le h\le (k-1)(n-1)$, we have
$$N(k,n,h)=\sum_{\ell=0}^h (-1)^{h-\ell}\binom{kn+1}{h-\ell}\prod_{i=0}^{n-1}\prod_{j=0}^{k-1}\frac{i+j+1+\ell}{i+j+1}.$$
\end{theorem}

Perhaps surprisingly, the generalized Narayana numbers have the following symmetry, which generalizes equation~\eqref{eq:symmetryN2}.

\begin{theorem}[{\cite[Cor.~1]{sulanke_generalizing_2004}}]\label{thm:symmetry}
For all $0\le h\le (k-1)(n-1)$, we have
\[ N(k,n,h)=N(k,n,(k-1)(n-1)-h). \]
\end{theorem}

In~\cite{sulanke_generalizing_2004}, Sulanke deduces this symmetry from Theorem~\ref{thm:Narayana-formula} after some manipulations. However, as we will discuss in Section~\ref{sec:posets}, this is in fact a special case of a result of Stanley about descents on linear extensions of posets.
Interestingly, the existing proofs of these symmetries are not bijective. 
The goal of this paper is to provide a bijective proof of Theorem~\ref{thm:symmetry}, and of other related identities. 

In Section~\ref{sec:main} we state the main results of the paper.
For standard Young tableaux of rectangular shape, these include bijective proofs of the symmetry of the distribution of $\asc$ and $\des$, as well as an involution with relates these two statistics, reproving a result of Sulanke~\cite{sulanke_generalizing_2004}.
For standard Young tableaux of (truncated) staircase shape, we have a bijective proof of the symmetry of the distribution of $\asc$.

In Section~\ref{sec:bij} we describe our main bijections, which rely on certain encodings of standard Young tableaux via arrows that indicate the placement of the entries. 
In Section~\ref{sec:proofs} we show that our bijections have the desired behavior with respect to the number of ascents and descents, proving the theorems from Section~\ref{sec:main}. In Section~\ref{sec:properties} we describe other interesting properties of our bijections for rectangular shapes, 
we relate them to the Lalanne--Kreweras involution in the case of two rows, and we introduce a map that extends {\em rowmotion} \cite{striker_promotion_2012} from Dyck paths (viewed as order ideals of the type $A$ root poset) to $\SYT(n^k)$.

In Section~\ref{sec:refined}, we consider the behavior of our bijections for rectangular tableaux on a collection of more refined descent statistics. In the special case of three-row tableaux, this allows us to prove a conjecture of Sulanke~\cite{sulanke_three_2005}. 
We also describe a family of $k!$ statistics on $\SYT(n^k)$ that have a generalized Narayana distribution, and use our construction to give a bijective proof of the symmetry of the number of descents on certain multiset permutations called canon permutations~\cite{elizalde_descents_2024}. 
In Section~\ref{sec:posets}, we explain how our work fits into the more general context of $P$-partitions and linear extensions of posets.
Finally, we state some open questions in Section~\ref{sec:open}, along with directions for further research.

\section{Main results}\label{sec:main}

In order to prove Theorem~\ref{thm:symmetry} bijectively, using the interpretation~\eqref{eq:Ndes} of the generalized Narayana numbers in terms of descents, it suffices to construct a bijection between tableaux in $\SYT(n^k)$ with $d=h+k-1$ descents, and tableaux with $(k-1)(n+1)-d=(k-1)(n-1)-h+k-1$ descents. 

\begin{theorem}\label{thm:bijd}
There exists an involution $\bijd:\SYT(n^k)\to\SYT(n^k)$ such that, for all $T\in\SYT(n^k)$,
$$
\des(T)+\des(\bijd(T))=(k-1)(n+1).
$$
\end{theorem}

Alternatively, if use the original interpretation~\eqref{eq:Nasc} of the generalized Narayana numbers in terms of ascents, then a 
direct bijective proof of Theorem~\ref{thm:symmetry} is the following. 

\begin{theorem}\label{thm:bija}
There exists an involution $\bija:\SYT(n^k)\to\SYT(n^k)$ such that, for all $T\in\SYT(n^k)$,
$$
\asc(T)+\asc(\bija(T))=(k-1)(n-1).
$$
\end{theorem}

The equivalence of the two definitions~\eqref{eq:Nasc} and~\eqref{eq:Ndes} of the generalized Narayana numbers is a consequence of a bijection of Sulanke~\cite[Prop.~2]{sulanke_generalizing_2004} between the set of tableaux in $\SYT(n^k)$ with $h$ ascents and those with $h+k-1$ descents. His construction is a composition of $k-2$ bijections, and it is not an involution in general. We will give another bijection between the same sets. Even though both bijections are similar in spirit, ours has two advantages: it is an involution, and it has a simpler description in terms of arrow encodings. 

\begin{theorem}\label{thm:rev}
There exists an involution $\rev:\SYT(n^k)\to\SYT(n^k)$ such that, for all $T\in\SYT(n^k)$,
$$\des(T)=\asc(\rev(T))+k-1 \quad\text{and}\quad \asc(T)+k-1=\des(\rev(T)).$$
\end{theorem}

The last result of this section is not about rectangular shapes, bur rather about standard Young tableaux of shape $\lambda=(n,n-1,\dots,n-k+1)$ for some $1\le k\le n$. We call this a {\em truncated staircase} shape, noting that, when $k=n$, the shape $\lambda=(n,n-1,\dots,1)$ is a staircase.

\begin{theorem}\label{thm:staircase}
Let $\lambda=(n,n-1,\dots,n-k+1)$, where $1\le k\le n$. There exists an involution $\bija:\SYT(\lambda)\to\SYT(\lambda)$ such that, for all $T\in\SYT(\lambda)$,
\begin{equation}\label{eq:bij_asc}
\asc(T)+\asc(\bija(T))=\frac{(2n-k)(k-1)}{2}.
\end{equation}
\end{theorem}

\section{The bijections}\label{sec:bij}

In this section we describe the main bijections of the paper, which are used to prove the theorems in Section~\ref{sec:main}. We start by recalling two simple operations on standard Young tableaux that will be used in Sections~\ref{sec:properties} and~\ref{sec:refined}.

\subsection{Two basic operations on tableaux}\label{sec:basic}

The {\em conjugate} of $T\in\SYT(\lambda)$ is the tableau $T'\in\SYT(\lambda')$ obtained by reflecting $T$ along the main diagonal. This operation switches rows with columns, so that cell $(r,c)$ of $T$ becomes cell $(c,r)$ of $T'$. 
For $1\le i\le N-1$, we have $i\in\Des(T)$ if and only if $i\notin\Des(T')$, and so 
\begin{equation}\label{eq:desT'} \des(T)+\des(T')=N-1. \end{equation} 
The shape $\lambda'$ is called the conjugate of $\lambda$. For example, the conjugate of the rectangular shape $\lambda=(n^k)$ is $\lambda'=(k^n)$, whereas the staircase $\lambda=(n,n-1,\dots,1)$ is {\em self-conjugate}, meaning that $\lambda=\lambda'$. It follows from equation~\eqref{eq:desT'} that the distribution of $\des$ over self-conjugate shapes is symmetric. 
However, this is not the case for truncated staircases in general: already for $\lambda=(3,2)$, it is easy to check that the distribution of $\des$ is not symmetric over $\SYT(\lambda)$. In particular, there is no analogue of Theorem~\ref{thm:staircase} for the statistic $\des$.

For a rectangular tableau $T\in\SYT(n^k)$, define its {\em rotation} to be the tableau $\rot(T)\in\SYT(n^k)$ obtained by rotating $T$ by $180$ degrees, so that cell $(r,c)$ becomes cell $(k+1-r,n+1-c)$, and replacing each entry $i$ with $kn+1-i$. As shown in~\cite{butler_subgroup_1994}, rotation of tableaux of rectangular shape coincides with {\em evacuation}, a map introduced by Sch\"utzenberger~\cite{schutzenberger_quelques_1963}, which arises in connection with the Robinson--Schensted correspondence, and is defined on standard Young tableaux of any shape.
For $1\le i\le kn-1$, we have that $i\in\Des(T)$ if and only if $kn-i\in\Des(\rot(T))$, and similarly
$i\in\Asc(T)$ if and only if $kn-i\in\Asc(\rot(T))$. It follows that $\des(T)=\des(\rot(T))$ and $\asc(T)=\asc(\rot(T))$.

Both conjugation and rotation are involutions, in the sense that $T''=T$ and $\rot(\rot(T))=T$.

\subsection{Arrow encodings}\label{sec:arrows}

To describe our main bijections, we will encode tableaux in $\SYT(\lambda)$ using arrows describing the placement of the numbers $1,2,\dots,N$.
We assign a (possibly empty) sequence of arrows in $\{\ua,\da\}$ to the right border of each cell, and also to the left border of the cells in the first column.

For $1\le r\le k$ and $1\le c\le \lambda_r$, let $(r,c)$ be the cell in row $r$ and column $c$. We will denote by $A_{r,c}$ the sequence of arrows assigned to the right border of the cell $(r,c)$ (if $c<\lambda_r$, this is also the left border of the cell $(r,c+1)$), and by $A_{r,0}$ the sequence of arrows assigned to the left border of the cell $(r,1)$. 
Denote by $I(\lambda)=\{(r,c): 1\le r\le k, 0\le c\le \lambda_r\}$ the set of possible indices.
An array $\{A_{r,c}\}_{(r,c)\in I(\lambda)}$ is called an {\em arrow array}.
For example, the arrow array on the left of Figure~\ref{fig:arrows} has $A_{2,3}=\da\!\ua\!\da$.

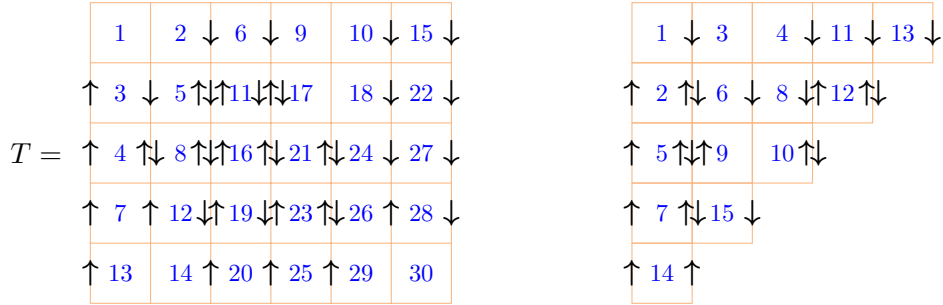
\begin{figure}[htb]
\centering
\begin{tikzpicture}[scale=.8]
\drawtab{6}{5}{{{1,2,6,9,10,15},{3,5,11,17,18,22},{4,8,16,21,24,27},{7,12,19,23,26,28},{13,14,20,25,29,30}}}
\drawarr{5}{{{,,\da,\da,,\da,\da},{\ua,\da,\ua\!\da\!\ua,\da\!\ua\!\da,,\da,\da},{\ua,\ua\!\da,\ua\!\da\!\ua,\ua\!\da,\ua\!\da,\da,\da},{\ua,\ua,\da\!\ua,\da\!\ua,\ua\!\da,\ua,\da},{\ua,,\ua,\ua,\ua,,}}}
\draw (-.3,2.5) node[left] {$T=$};
\begin{scope}[shift={(9,0)}]
\drawtabgen{5}{{{1,3,4,11,13},{2,6,8,12},{5,9,10},{7,15},{14}}}
\drawarr{5}{{{,\da,,\da,\da,\da},{\ua,\ua\!\da,\da,\da\!\ua,\ua\!\da},{\ua,\ua\!\da\!\ua,,\ua\!\da},{\ua,\ua\!\da,\da},{\ua,\ua}}}
\end{scope}
\end{tikzpicture}
\caption{The arrow encodings of two standard Young tableaux of rectangular and staircase shapes.}
\label{fig:arrows}
\end{figure}

Let us introduce some more notation that we need to define the encoding. For $1\le i\le N$, denote by $\row(i)$ the index of the row where $i$ appears in $T$.
Additionally, denote by $T^{\le i}$ be the standard Young tableau consisting of the entries less than or equal to $i$ in $T$, and let $(r,\max_r^i)$ be the rightmost filled cell in row $r$ of $T^{\le i}$.

\begin{definition}\label{def:arrow_encoding}
The {\em arrow encoding} of $T\in\SYT(\lambda)$ is the arrow array $\{A_{r,c}\}_{(r,c)\in I(\lambda)}$, where each $A_{r,c}$ is a sequence of arrows in $\{\ua,\da\}$ constructed as follows.
Initially, set $A_{r,0}=\ua$ for $2\le r\le k$, and $A_{r,c}=\es$ (the empty sequence) for all other cells $(r,c)$. Next, for each $i$ from $1$ to $N-1$:
\begin{itemize}
\item if $\row(i)=\row(i+1)$, do nothing;
\item if $\row(i)<\row(i+1)$, then for each $\row(i)\le r<\row(i+1)$, append a $\da$ to $A_{r,\max_r^i}$;
\item if $\row(i)>\row(i+1)$, then for each $\row(i)\ge r>\row(i+1)$, append a $\ua$ to $A_{r,\max_r^i}$.
\end{itemize}
Finally, for each $\row(N)\le r<k$, append a $\da$ to $A_{r,\lambda_r}$.
\end{definition}
One can think of these arrows as describing how to reach the row where $i+1$ will be placed from the row where $i$ has been placed. Figure~\ref{fig:arrows} shows two examples. Next, we will characterize which arrow arrays are obtained by encoding tableaux in $\SYT(\lambda)$.

\begin{definition}\label{def:valid}
An arrow array $\{A_{r,c}\}_{(r,c)\in I(\lambda)}$ is {\em valid} if the following conditions hold:
\begin{itemize}
\item {\it Left Boundary (LB):} We have $A_{1,0}=\es$, and $A_{r,0}=\ua$ for all $2\le r\le k$.
\item {\it Right Boundary (RB):} For all $1\le r\le k-1$, the sequence $A_{r,\lambda_r}$ ends with a $\da$.
\item {\it Alternation:} For all $(r,c)\in I(\lambda)$, the sequence $A_{r,c}$ does not contain two consecutive $\ua$ or two consecutive~$\da$.
\item {\it Matching:} For all $1\le r\le k-1$, the total number of $\da$ in row $r$ equals the total number of $\ua$ in row $r+1$. There are no $\ua$ in row $1$ and no $\da$ in row~$k$.
\item {\it Ballot:} For all $1\le r\le k-1$ and $1\le c\le \min\{\lambda_r-1,\lambda_{r+1}\}$, the total number of $\da$ in $A_{r,1},A_{r,2},\dots,A_{r,c}$ is at most the total number of $\ua$ in $A_{r+1,1},A_{r+1,2},\dots,A_{r+1,c}$. 
\end{itemize}
\end{definition}

Note that requiring $A_{1,\lambda_1}$ to end with a $\da$ is equivalent to requiring $A_{1,\lambda_1}=\da$, due to the Alternation condition and the fact that there are no $\ua$ in row $1$. Before we prove that these conditions are necessary and sufficient, let us show that, in the special case of rectangles, we can replace the RB condition with a simpler one.

\begin{lemma}\label{lem:valid_rectangles} 
In the case $\lambda=(n^k)$, an arrow array $\{A_{r,c}\}_{1\le r\le k,0\le c\le n}$ is valid if and only if it satisfies the conditions in Definition~\ref{def:valid} with the Right Boundary condition replaced with:
\begin{itemize}
\item {\it Right Boundary for Rectangles (RBR)}: We have $A_{r,n}=\da$ for all $1\le r\le k-1$, and $A_{k,n}=\es$. 
\end{itemize}
\end{lemma}

\begin{proof}
Let $\lambda=(n^k)$. The RBR condition clearly implies the RB condition. Thus, it suffices to show that every valid arrow array $\{A_{r,c}\}$, as in Definition~\ref{def:valid}, also satisfies the RBR condition. Suppose that this is not the case. Using the Alternating condition and the fact that there are no $\da$ in row $k$, there must be some $r$ such that $A_{r,n}$ contains a $\ua$. Let $r$ be the smallest such index.
Since $A_{1,n}=\da$, we have $r\ge2$, so $A_{r-1,n}=\da$ by minimality of $r$. By the Matching condition, the total number of $\da$ in row $r-1$ (call this number $t$) equals the total number of $\ua$ in row $r$. The number of $\da$ in $A_{r-1,1},A_{r-1,2},\dots,A_{r-1,n-1}$ is then $t-1$, since $A_{r-1,0}$ does no have any $\da$ by the LB condition, and $A_{r-1,n}=\da$. Similarly, the number of $\ua$ in $A_{r,1},A_{r,2},\dots,A_{r,n-1}$ is at most $t-2$, since $A_{r,0}=\ua$ by the LB condition, and $A_{r,n}$ contains a $\ua$. But this contradicts the Ballot condition.
\end{proof}

\begin{lemma}\label{lem:encoding}
For any partition $\lambda$, the arrow encoding from Definition~\ref{def:arrow_encoding} is a bijection between tableaux in $\SYT(\lambda)$ and valid arrow arrays $\{A_{r,c}\}_{(r,c)\in I(\lambda)}$.
\end{lemma}

\begin{proof}
It is not hard to see from Definition~\ref{def:arrow_encoding} that the arrow encoding of any $T\in\SYT(\lambda)$ satisfies the five conditions in Definition~\ref{def:valid}, so it is a valid array.

Next we describe the inverse map, which tells us how to recover the tableau $T$ from any given valid arrow array $\{A_{r,c}\}_{(r,c)\in I(\lambda)}$. Our algorithm will read the arrows from left to right within each row, using them to jump between adjacent rows and place the entries $1,2,\dots,N$ to form $T$.

We start by reading the arrows $\ua$ in $A_{r,0}$ for $2\le r\le k$, which bring us up to the first row, and by placing the entry $1$ in cell $(1,1)$. 
Now, for each $i$ from $1$ to $N-1$, suppose that the first $i$ entries have already been placed, forming $T^{\le i}$, and let $(r,c)$ be the cell containing entry~$i$. 

Let us describe where to place entry $i+1$.  If $A_{r,c}$ is empty,  place $i+1$ in cell $(r,c+1)$. Otherwise, read the first arrow in $A_{r,c}$. 
\begin{itemize}
\item If it is a $\ua$, jump up to row $r-1$. If $A_{r-1,\max_{r-1}^i}$ has no unread arrows, place $i+1$ in cell $(r-1,\max_{r-1}^i+1)$. 
Otherwise, read the first unread arrow of $A_{r-1,\max_{r-1}^i}$ (which must be a $\ua$ because of the Alternation condition in Definition~\ref{def:valid}), jump up to row $r-2$ and repeat this process, until we reach a row $r'$ such that $A_{r',\max_{r'}^i}$ has no unread arrows, at which point we place $i+1$ in cell $(r',\max_{r'}^i+1)$. 
\item If it is a $\da$, jump down to row $r+1$ and proceed analogously, but in the downward direction.
\end{itemize}

Let us show that this algorithm always produces a tableau in $\SYT(\lambda)$. It is clear that the rows of the resulting tableau $T$ increase from left to right. 

To show that the columns increase from top to bottom, suppose for contradiction that $i+1$ is placed in cell $(r',c')$, where $2\le r'\le k$ and $1\le c'\le \lambda_{r'}$, and that the cell $(r'-1,c')$ is empty in $T^{\le i}$. In this case, from the placement of $1$ until the placement of $i+1$, the number of jumps from row $r'-1$ to row $r'$ would be greater than the number of jumps from row $r'$ to row $r'-1$, since all the jumps are between adjacent rows. This means that the number of $\da$ in 
$A_{r'-1,1},A_{r'-1,2},\dots,A_{r'-1,c'-1}$ would be greater than the number of $\ua$ in $A_{r',1},A_{r',2},\dots,A_{r',c'-1}$, contradicting the Ballot condition in Definition~\ref{def:valid}. 

Finally, let us show that the resulting tableau has shape $\lambda$, by arguing that the above procedure never places an entry outside of this shape. It is clear that no entry will be placed below row $k$, since there are no $\da$ in this row.
To show that no entry will be placed to the right of the shape $\lambda$, suppose for the sake of contradiction that this is not true, and let $i+1$ be the first entry that the algorithm places in a cell of the form $(r',\lambda_{r'}+1)$, for some $1\le r'\le k$.

Suppose first that $r'<k$. From the placement of $1$ until the placement of $i+1$, the number of jumps from row $r'$ to row $r'+1$ (call this number $t$) must be the same as the number of jumps from row $r'+1$ to row $r'$. We know that $t$ equals the total number of $\da$ in row $r'$, which, by the Matching condition in Definition~\ref{def:valid}, equals the total number of $\ua$ in row $r'+1$. But this includes the $\ua$ in $A_{r'+1,0}$, which did not contribute to a jump from row $r'+1$ to row $r'$. Thus, the number of jumps from row $r'+1$ to row $r'$ must have been strictly less than $t$, reaching a contradiction.

Suppose now that $r'=k$. We will show that the $N$ cells of the shape $\lambda$ are already full before $i+1$ is placed. This is clearly the case if $k=1$, so suppose that $k\ge2$. When $i+1$ is placed in row $k$, the number of $\da$ that have been read so far in row $r$ must equal the number of $\ua$ that have been read in row $r+1$ (including the $\ua$ in $A_{r+1,0}$), for each $1\le r\le k-1$. And since $i+1$ is being placed in $(r,\lambda_r+1)$, all the arrows in row $r$ have been read at this point. By the Matching condition, the total number of $\ua$ in row $r$ equals the total number of $\da$ in row $r-1$, so we deduce that all the $\da$ in row $r-1$ have been read, hence so have all the arrows in row $r-1$, since the rightmost arrow in this row is a $\da$ by the RB condition.
By repeating the same argument, we deduce that all the arrows in rows $r-2,r-3,\dots,1$ have been read as well. We conclude that all the arrows of the arrow array have been read, hence all the cells of $\lambda$ are already full.
Therefore, in this case, the algorithm actually terminates before placing $i+1$.
\end{proof}

\subsection{Bijections for rectangular shapes}\label{sec:bij_rect}

We will define the bijection $\bijd$ that proves Theorem~\ref{thm:bijd} as a composition of certain involutions $\bijd_r$, each one affecting arrows in adjacent rows of the arrow encoding.

\begin{definition}\label{def:bijd}
For $1\le r\le k-1$, let $\bijd_r:\SYT(n^k)\to\SYT(n^k)$ be the map that sends $T$ to the tableau $\bijd_r(T)$ whose arrow encoding is obtained from the arrow encoding $\{A_{r,c}\}$ of $T$ as follows. For every $0\le c\le n$, 
\begin{enumerate}[label=(\roman*)]
\item if $A_{r,c}$ starts with $\da$ and $A_{r+1,c}$ ends with $\ua$, remove these two arrows;
\item if $A_{r,c}$ does not start with $\da$ and $A_{r+1,c}$ does not end with $\ua$, insert $\da$ at the beginning of $A_{r,c}$ and $\ua$ at the end of $A_{r+1,c}$.
\end{enumerate}
Let $\bijd:\SYT(n^k)\to\SYT(n^k)$ be the composition $\bijd=\bijd_{1}\circ\bijd_2\circ\dots\circ\bijd_{k-1}$.
\end{definition}

\begin{figure}[htb]
\centering
\begin{tikzpicture}[scale=.8]
\drawtab{6}{5}{{{1,2,6,9,10,15},{3,5,11,17,18,22},{4,8,16,21,24,27},{7,12,19,23,26,28},{13,14,20,25,29,30}}}
\drawarr{5}{{{,,\da,\da,,\da,\da},{\ua,\da,\ua\!\da\!\ua,\da\!\ua\!\da,,\da,\da},{\ua,\ua\!\da,\ua\!\da\!\ua,\ua\!\da,\ua\!\da,\dar,\dar},{\uar,\uar,\da\!\uar,\da\!\uar,\ua\!\da,\uar,\da},{\ua,,\ua,\ua,\ua,,}}}
\draw[<->] (7,2.5)--node[above]{$\textcolor{red}{\bijd_3}$} (8,2.5);
\draw (-.3,2.5) node[left] {$T=$};
\begin{scope}[shift={(9,0)}]
\drawtab{6}{5}{{{1,2,6,9,10,15},{3,5,11,17,18,23},{4,8,16,21,25,26},{7,12,19,22,27,28},{13,14,20,24,29,30}}}
\drawarr{5}{{{,,\da,\da,,\da,\da},{\ua,\da,\ua\!\da\!\ua,\da\!\ua\!\da,,\da,\da},{\ua,\ua\!\da,\ua\!\da\!\ua,\ua\!\da,\dar\!\ua\!\da,,\dar},{\uar,\uar,\da\!\uar,\da\!\uar,\ua\!\da\!\uar,,\da},{\ua,,\ua,\ua,\ua,,}}}
\draw (6.3,2.5) node[right] {$=\bijd_3(T)$};
\end{scope}
\end{tikzpicture}
\caption{An example of the involution $\bijd_r$ on $\SYT(6^5)$ for $r=3$. Each leading $\da$ in row $r$ and each trailing $\ua$ in row $r+1$ is colored in red. In columns with a pair of red arrows $\begin{array}{c}\dar\\\uar\end{array}$, the map $\bijd_r$ removes them, and in columns with no red arrows, the map adds such a pair.}
\label{fig:bijd_r}
\end{figure}

An arrow at the beginning of $A_{r,c}$ will be called a {\em leading} arrow, and an arrow at the end of $A_{r+1,c}$ will be called a {\em trailing} arrow. An example of $\bijd_r$ appears in Figure~\ref{fig:bijd_r}, and an example of $\bijd$ appears at the top of Figure~\ref{fig:bij}.

\begin{figure}[htb]
\centering
\begin{tikzpicture}[scale=.8]
\drawtab{6}{5}{{{1,2,6,9,10,15},{3,5,11,17,18,22},{4,8,16,21,24,27},{7,12,19,23,26,28},{13,14,20,25,29,30}}}
\drawarr{5}{{{,,\darrr,\darrr,,\darrr,\darrr},{\uarrr,\darrrr,\ua\!\da\!\uarrr,\darrrr\!\ua\!\da,,\darrrr,\darrrr},{\uarrrr,\ua\!\da,\ua\!\da\!\uarrrr,\ua\!\da,\ua\!\da,\dar,\dar},{\uar,\uar,\darr\!\uar,\darr\!\uar,\ua\!\da,\uar,\darr},{\uarr,,\uarr,\uarr,\uarr,,}}}
\draw[<->] (7,2.5)--node[above]{$\bijd=\textcolor{violet}{\bijd_1}\circ\textcolor{cyan}{\bijd_2}\circ\textcolor{red}{\bijd_3}\circ\textcolor{teal}{\bijd_4}$} (11,2.5);
\draw (-.3,2.5) node[left] {$T=$};
\begin{scope}[shift={(12,0)}]
\drawtab{6}{5}{{{1,4,5,7,13,19},{2,6,11,15,20,24},{3,10,14,17,25,26},{8,12,16,18,27,29},{9,21,22,23,28,30}}}
\drawarr{5}{{{,\darrr,,\darrr,\darrr,\darrr,\darrr},{\uarrr,\darrrr\!\uarrr,\ua\!\da,\darrrr\!\ua\!\da,\darrrr\!\uarrr,\darrrr,\darrrr},{\uarrrr,\ua\!\da,\ua\!\da\!\uarrrr,\ua\!\da,\dar\!\ua\!\da\!\uarrrr,,\dar},{\uar,\darr\!\uar,\uar,\uar,\ua\!\da\!\uar,\darr,\darr},{\uarr,\uarr,,,\uarr,\uarr,}}}
\draw[<->] (3,-.5)--node[right]{$\rev$} (3,-1.1);
\end{scope}
\draw[<->] (3,-.5)--node[right]{$\rev$} (3,-1.1);
\begin{scope}[shift={(0,-6.5)}]
\drawtab{6}{5}{{{1, 2, 7, 9, 10, 13},{3, 6, 11, 18, 19, 24},{4, 8, 14, 20, 25, 27},{5, 12, 17, 22, 26, 28},{15, 16, 21, 23, 29, 30}}}
\drawarr{5}{{{,,\darrr,\darrr,,\darrr,\darrr},{\uarrr,\darrrr,\uarrr\!\da\!\ua,\da\!\ua\!\darrrr,,\darrrr,\darrrr},{\uarrrr,\da\!\ua,\uarrrr\!\da\!\ua,\da\!\ua,\da\!\ua,\dar,\dar},{\uar,\uar,\uar\!\darr,\uar\!\darr,\da\!\ua,\uar,\darr},{\uarr,,\uarr,\uarr,\uarr,,}}}
\draw[<->] (7,2.5)--node[above]{$\bija=\textcolor{violet}{\bija_1}\circ\textcolor{cyan}{\bija_2}\circ\textcolor{red}{\bija_3}\circ\textcolor{teal}{\bija_4}$} (11,2.5);
\begin{scope}[shift={(12,0)}]
\drawtab{6}{5}{{{1, 3, 4, 9, 13, 17},{2, 7, 10, 16, 19, 21},{5, 8, 14, 18, 25, 26},{6, 12, 15, 20, 27, 29},{11, 22, 23, 24, 28, 30}}}
\drawarr{5}{{{,\darrr,,\darrr,\darrr,\darrr,\darrr},{\uarrr,\uarrr\!\darrrr,\da\!\ua,\da\!\ua\!\darrrr,\uarrr\!\darrrr,\darrrr,\darrrr},{\uarrrr,\da\!\ua,\uarrrr\!\da\!\ua,\da\!\ua,\uarrrr\!\da\!\ua\!\dar,,\dar},{\uar,\uar\!\darr,\uar,\uar,\uar\!\da\!\ua,\darr,\darr},{\uarr,\uarr,,,\uarr,\uarr,}}}
\end{scope}
\end{scope}
\end{tikzpicture}
\caption{Examples of the involutions $\bijd$, $\rev$, and $\bija$. In the top tableaux, all the leading $\da$ in row $r$ and the trailing $\ua$ in row $r+1$ have the same color for each fixed $r$, to help visualize $\bijd$. After applying $\rev$, these become trailing $\da$ and leading $\ua$, respectively.}
\label{fig:bij}
\end{figure}

The bijection $\bija$ that proves Theorem~\ref{thm:bija} has a similar description as a composition of involutions $\bija_r$, which are defined analogously to $\bijd_r$ but where the roles of leading and trailing arrows are switched.

\begin{definition}\label{def:bija}
For $1\le r\le k-1$, let $\bija_r:\SYT(n^k)\to\SYT(n^k)$ be the map that sends $T$ to the tableau $\bija_r(T)$ whose arrow encoding is obtained from the arrow encoding $\{A_{r,c}\}$ of $T$ as follows. For every $0\le c\le n$, 
\begin{enumerate}[label=(\roman*)]
\item if $A_{r,c}$ ends with $\da$ and $A_{r+1,c}$ starts with $\ua$, remove these two arrows;
\item if $A_{r,c}$ does not end with $\da$ and $A_{r+1,c}$ does not start with $\ua$, insert $\da$ at the end of $A_{r,c}$ and $\ua$ at the beginning of $A_{r+1,c}$.
\end{enumerate}
Let $\bija:\SYT(n^k)\to\SYT(n^k)$ be the composition $\bija=\bija_1\circ\bija_2\circ\dots\circ\bija_{k-1}$. 
\end{definition}

An example of $\bija$ appears at the bottom of Figure~\ref{fig:bij}.
Note that the maps $\bijd_r$ and $\bija_r$ never make any changes to the sequences $A_{r,c}$ and $A_{r+1,c}$ for $c\in\{0,n\}$. Indeed, the LB condition from Definition~\ref{def:valid} guarantees that $A_{r,0}$ does not contain a $\da$, whereas $A_{r+1,0}=\ua$, so these sequences remain unchanged. Similarly, because of the RBR condition from Lemma~\ref{lem:valid_rectangles}, we have $A_{r,n}=\da$, whereas $A_{r+1,n}$ does not contain a $\ua$, so these sequences do not change either.

\begin{lemma}\label{lem:involution}
The maps $\bijd$ and $\bija$ are well-defined involutions on $\SYT(n^k)$.
\end{lemma}

\begin{proof}
Let us first show that the maps $\bijd_r$, and hence $\bijd$, are well defined, by checking that the arrow array obtained after applying the transformation in Definition~\ref{def:bijd} is valid. The Left and Right Boundary conditions in Definition~\ref{def:valid} hold since, as shown in the above paragraph, the sequences $A_{r,c}$ and $A_{r+1,c}$ for $c\in\{0,n\}$ are not changed by $\bijd_r$.
The Alternation condition holds because $\da$ can only be inserted at the beginning of a sequence with no leading $\da$, and $\ua$ can only be appended to a sequence with no trailing $\ua$. Finally, the Matching and the Ballot conditions hold because, when applying $\bijd_r$, the number of $\da$ in $A_{r,c}$ always changes by the same amount as the number of $\ua$ in $A_{r+1,c}$.

It is clear from Definition~\ref{def:bijd} that the maps $\bijd_r$ are involutions on $\SYT(n^k)$, so in particular they are bijections. Additionally, these maps commute with each other for different values of $r$. This is because $\bijd_r$ only affects leading $\da$ in row $r$ and trailing $\ua$ in row $r+1$, so the changes made by different $\bijd_r$ do not interact.
It follows that $$\bijd^2=(\bijd_1\circ\bijd_2\circ\dots\circ\bijd_{k-1})\circ(\bijd_1\circ\bijd_2\circ\dots\circ\bijd_{k-1})=\bijd_1^2\circ\bijd_2^2\circ\dots\bijd_{k-1}^2,$$
which equals the identity since the $\bijd_r$ are involutions. Hence, $\bijd$ is an involution.

The fact that $\bija_r$ and $\bija$ are well-defined involutions is proved analogously.
\end{proof}

The argument in the previous proof shows that one can compute $\bijd(T)$ (respectively $\bija(T)$) by applying all the maps $\bijd_r$ (resp.~$\bija_r$), for $1\le r\le k-1$, at the same time, as illustrated in Figure~\ref{fig:bij}. 
The behavior of these bijections with respect to the statistics $\des$ and $\asc$ will be studied in Section~\ref{sec:proofs}.

The bijections $\bijd$ and $\bija$ on $\SYT(n^k)$ are related to each other via another involution $\rev$, which will be used to prove Theorem~\ref{thm:rev}. In Figure~\ref{fig:bij}, $\rev$ is the map going between the top and the bottom tableaux.

\begin{definition}\label{def:rev}
Let $\rev:\SYT(n^k)\to\SYT(n^k)$ be the map that sends $T\in\SYT(n^k)$ to the tableau $\rev(T)$ whose arrow encoding is obtained from the arrow encoding $\{A_{r,c}\}$ of $T$ by reversing (i.e., reading from right to left) each sequence $A_{r,c}$.
\end{definition}

\begin{lemma}\label{lem:involution_rev}
The map $\rev$ is a well-defined involution on $\SYT(n^k)$.
\end{lemma}

\begin{proof}
Using Definition~\ref{def:valid} with the variation for rectangles given in Lemma~\ref{lem:valid_rectangles}, an arrow array $\{A_{r,c}\}_{1\le r\le k,0\le c\le n}$ is valid if and only if so is the arrow array obtained by reversing each sequence $A_{r,c}$. Therefore, $\rev$ is well-defined. Additionally, it is clearly an involution.
\end{proof}

Since $\rev$ takes leading arrows to trailing arrows and vice versa, we see from Definitions~\ref{def:bijd}, \ref{def:bija} and \ref{def:rev} that $\bija_r=\rev\circ\bijd_r\circ\rev$ for all $r$, and 
\begin{equation}\label{eq:bijad} 
\bija=\rev\circ\bijd\circ\rev,
\end{equation}
as illustrated in Figure~\ref{fig:bij}.

\subsection{Bijections for truncated staircase shapes}\label{sec:bij_stair}

The definitions of $\bijd$, $\bija$ and $\rev$ in Section~\ref{sec:bij_rect} rely on the rectangular shape of the standard Young tableaux. In this section we discuss possible generalizations to other shapes, and we show that there is a useful extension of the definition of $\bija$ to $\SYT(\lambda)$ where $\lambda$ is any partition into distinct parts.

Let us start with the bijection $\rev$. Definition~\ref{def:rev} does not readily generalize to shapes other than rectangles, since the RB condition in Definition~\ref{def:valid} is not necessarily preserved when reversing the arrows in $A_{r,\lambda_r}$.
One could modify the definition to avoid changing the sequences $A_{r,\lambda_r}$, but then the bijection would lose its useful behavior with respect to ascents and descents.

For the maps $\bijd_r$ and $\bija_r$, one could naively try to generalize Definitions~\ref{def:bijd} and~\ref{def:bija} to non-rectangular shapes by simply replacing the quantifier ``for every $0\le c\le n$'' with ``for every $0\le c\le \lambda_{r+1}$,'' so that the relevant sequences $A_{r+1,c}$ are defined. However, there is no guarantee that the resulting arrow arrays would be valid in general. Specifically, the map $\bijd_r$ could be adding a trailing $\ua$ to $A_{r+1,\lambda_{r+1}}$ that violates the RB condition from Definition~\ref{def:valid}, as shown on the left of Figure~\ref{fig:problems}.
This problem happens whenever $A_{r,\lambda_{r+1}}$ does not have a leading~$\da$.

\begin{figure}[htb]
\centering
\begin{tikzpicture}[scale=.7]
\drawtabgen{3}{{{1,2},{3},{4}}}
\drawarr{3}{{{,,\darrr},{\uarrr,\da,},{\ua,,}}}
rrr\draw[<->] (2.5,1.5)--node[above]{$\textcolor{violet}{\bijd_1}$?} (3,1.5);
\begin{scope}[shift={(3.5,0)}]
\drawtabgen{3}{{{,},{ },{ }}}
\drawarr{3}{{{,\darrr,\darrr},{\uarrr,\da\!\uarrr,},{\ua,,}}}
\end{scope}
\end{tikzpicture}
\qquad\qquad
\begin{tikzpicture}[scale=.7]
\drawtabgen{3}{{{1,4},{2},{3}}}
\drawarr{3}{{{,\da,\da},{\ua,\da\!\ua\!\darrrr,},{\uarrrr,\uarrrr,}}}
\draw[<->] (2.5,1.5)--node[above]{$\textcolor{cyan}{\bija_2}$?} (3,1.5);
\begin{scope}[shift={(3.5,0)}]
\drawtabgen{3}{{{,},{ },{ }}}
\drawarr{3}{{{,\da,\da},{\ua,\da\!\ua,},{\uarrrr,,}}}
\end{scope}
\end{tikzpicture}
\caption{The definitions of $\bijd_r$ and $\bija_r$ fail to produce valid arrow arrays for the shape $\lambda=(2,1,1)$.}
\label{fig:problems}
\end{figure}
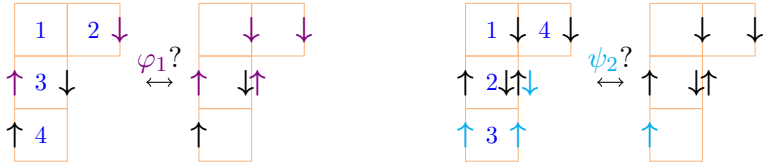

On the other hand, the map $\bija_r$ could be removing a trailing $\da$ from $A_{r,\lambda_r}$, again violating the RB condition, as shown on the right of Figure~\ref{fig:problems}. This happens whenever $\lambda_r=\lambda_{r+1}$ and $A_{r,\lambda_{r+1}}$ has a leading $\ua$. 
Aside from the case of rectangular shapes, where the RBR condition prevents this situation from happening, we can avoid this problem if we require $\lambda_r>\lambda_{r+1}$. In particular, if $\lambda$ has distinct parts (i.e., $\lambda_1>\lambda_2>\dots>\lambda_k$), then we can define $\bija$ as follows.

\begin{definition}\label{def:bija_distinct}
Let $\lambda=(\lambda_1,\lambda_2,\dots,\lambda_k)$ be a partition into distinct parts. For $1\le r\le k-1$, let $\bija_r:\SYT(\lambda)\to\SYT(\lambda)$ be the map that sends $T$ to the tableau $\bija_r(T)$ whose arrow encoding is obtained from the arrow encoding $\{A_{r,c}\}$ of $T$ by applying the operations (i) and (ii) from Definition~\ref{def:bija} for every $0\le c\le \lambda_{r+1}$.\\
Let $\bija:\SYT(\lambda)\to\SYT(\lambda)$ be the composition $\bija=\bija_1\circ\bija_2\circ\dots\circ\bija_{k-1}$. 
\end{definition}

An example of this bijection appears in Figure~\ref{fig:bija-stair}.

\begin{figure}[htb]
\centering
\begin{tikzpicture}[scale=.8]
\drawtabgen{5}{{{1,3,4,11,13},{2,6,8,12},{5,9,10},{7,15},{14}}}
\drawarr{5}{{{,\darrr,,\darrr,\darrr,\darrr},{\uarrr,\uarrr\!\darrrr,\darrrr,\da\!\ua,\uarrr\!\darrrr},{\uarrrr,\uarrrr\!\da\!\ua,,\uarrrr\!\dar},{\uar,\uar\!\darr,\darr},{\uarr,\uarr}}}
\draw[<->] (5.5,2.5)--node[above]{$\bija=\textcolor{violet}{\bija_1}\circ\textcolor{cyan}{\bija_2}\circ\textcolor{red}{\bija_3}\circ\textcolor{teal}{\bija_4}$} (9.5,2.5);
\begin{scope}[shift={(10.5,0)}]
\drawtabgen{5}{{{1,2,5,12,13},{3,4,8,14},{6,9,11},{7,10},{15}}}
\drawarr{5}{{{,,\darrr,\darrr,,\darrr},{\uarrr,,\uarrr\!\darrrr,\da\!\ua,\darrrr},{\uarrrr,\da\!\ua,\dar,\uarrrr\!\dar},{\uar,\uar,\uar\!\darr},{\uarr,}}}\end{scope}
\end{tikzpicture}
\caption{An example of the involution $\bija$ on a tableau of staircase shape $\lambda=(5,4,3,2,1)$. The trailing $\da$ in row $r$ and the leading $\ua$ in row $r+1$ have the same color for each fixed $r$.}
\label{fig:bija-stair}
\end{figure}
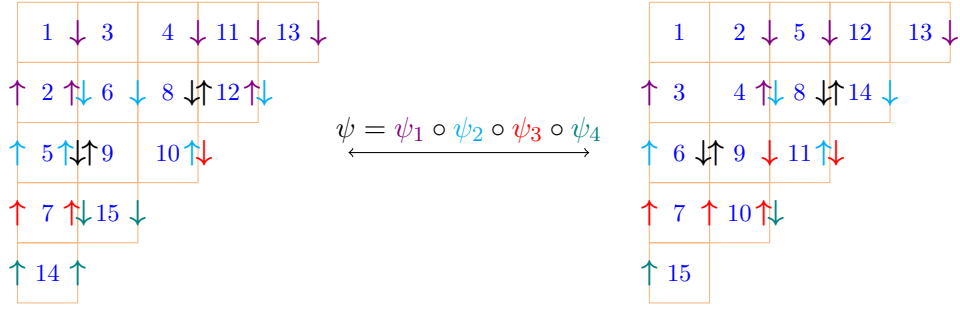

\begin{lemma}\label{lem:involution_distinct}
For any partition $\lambda$ into distinct parts, the map $\bija$ is well-defined involution on $\SYT(\lambda)$.
\end{lemma}

\begin{proof}
To see that $\bija_r$ and $\bija$ are well-defined, let us check that, if $\{A_{r,c}\}$ is the arrow encoding of $T\in\SYT(\lambda)$, then the arrow array obtained after applying the transformation in Definition~\ref{def:bija_distinct} is valid. The LB condition from Definition~\ref{def:valid} trivially holds because the sequences $A_{r,0}$ and $A_{r+1,0}$ do not change. Additionally, the RB condition holds because $A_{r,\lambda_r}$ is does not change (since $\bija_r$ can only affect $A_{r,c}$ for $0\le c\le \lambda_{r+1}$, but $\lambda_r>\lambda_{r+1}$), and any $\da$ in $A_{r+1,\lambda_{r+1}}$ are preserved.

As in the proof of Lemma~\ref{lem:involution}, the Alternation condition holds because $\da$ can only be inserted at the end of a sequence with no trailing $\da$, and $\ua$ can only be inserted at the beginning of a sequence with no leading $\ua$. The Matching and the Ballot conditions hold because, when applying $\bija_r$, the number of $\da$ in $A_{r,c}$ always changes by the same amount as the number of $\ua$ in $A_{r+1,c}$.

Finally, since the maps $\bija_r$ are involutions on $\SYT(\lambda)$ that commute with each other, their composition $\bija$ is an involution.
\end{proof}

Even though we have defined $\bija$ on $\SYT(\lambda)$ for any partition $\lambda$ into distinct parts, its behavior with respect to the number of ascents is not predictable in general. However, in the special case of truncated staircase shapes, we will show in Section~\ref{sec:properties} that $\bija$ behaves as described in Theorem~\ref{thm:staircase}.

\section{Proofs of the main theorems}\label{sec:proofs}

In this section, we prove the theorems in Section~\ref{sec:main} by analyzing the behavior of the bijections $\bijd$, $\bija$ and $\rev$, defined in Section~\ref{sec:bij}, with respect to the number of ascents and descents. 

Let us start by introducing some notation. Let $\lambda=(\lambda_1,\dots,\lambda_k)$ be a partition of $N$, and let $T\in\SYT(\lambda)$ with arrow encoding $\{A_{r,c}\}$. For $1\le r\le k$, define the following subsets of $\{0,1,\dots,\lambda_r\}$:
\begin{align*}
\Ld_{r}(T)&=\{c:A_{r,c}\text{ has a leading }\da\}, & \quad \Td_{r}(T)&=\{c:A_{r,c}\text{ has a trailing }\da\},\\
 \Lu_{r}(T)&=\{c:A_{r,c}\text{ has a leading }\ua\},
 & \Tu_{r}(T)&=\{c:A_{r,c}\text{ has a trailing }\ua\}.
\end{align*}
Note that, for $2\le r\le k$, both $\Tu_{r}(T)$ and $\Lu_{r}(T)$ contain $0$, because of the LB condition from Definition~\ref{def:valid}.
Additionally, $\lambda_r\in\Td_r(T)$ for $1\le r\le k-1$, because of the RB condition.
The next lemma relates the above sets to the number of ascents and descents of $T$.

\begin{lemma}\label{lem:desL}
For any $T\in\SYT(\lambda)$, we have
\begin{align*}
\des(T)&=\sum_{r=1}^{k-1}\card{\Ld_{r}(T)}=\sum_{r=2}^{k}\card{\Tu_{r}(T)},\\
\asc(T)+k-1&=\sum_{r=2}^{k}\card{\Lu_{r}(T)}=\sum_{r=1}^{k-1}\card{\Td_{r}(T)}.
\end{align*}
\end{lemma}

\begin{proof}
Let  $\{A_{r,c}\}$ be the arrow encoding of $T$, and suppose that entry $i$ is in cell $(r_i,c_i)$ and entry $i+1$ is in cell $(r_{i+1},c_{i+1})$.
Then $i$ is a descent of $T$ if and only if $A_{r_i,c_i}$ has a leading $\da$, since, in Definition~\ref{def:arrow_encoding}, this arrow indicates that $i+1$ is in a lower row (i.e., $r_{i+1}>r_i$). Noting that there are no $\da$ in $A_{r,c}$ if $r=k$ or $c=0$, it follows that 
$\des(T)=\sum_{r=1}^{k-1}\card{\Ld_{r}(T)}$.

Similarly, $i$ is an ascent of $T$ if and only if $A_{r_i,c_i}$ has a leading $\ua$. Noting that there are no $\ua$ in row $1$, and subtracting one for each $2\le r\le k$, since the arrows $A_{r,0}=\ua$ do not contribute to ascents, we obtain
$\asc(T)=\sum_{r=2}^{k}(\card{\Lu_{r}(T)}-1)$.

On the other hand, we claim that $i$ is a descent of $T$ if and only if $A_{r_{i+1},c_{i+1}-1}$ has a trailing $\ua$. This is clear if $c_{i+1}=1$, since in this case $A_{r_{i+1},0}=\ua$, and entry $i$ must be higher than $i+1$. Suppose now  that $c_{i+1}>1$. A trailing $\ua$ in $A_{r_{i+1},c_{i+1}-1}$ guarantees that, in the algorithm described in the proof of Lemma~\ref{lem:encoding}, the entry placed before $i+1$, namely $i$, lies in a higher row, so $i$ is a descent. Conversely, if $A_{r_{i+1},c_{i+1}-1}$ does not have a trailing $\ua$, then $i$ must have been placed in the same row $r_{i+1}$ (if $A_{r_{i+1},c_{i+1}-1}=\es$) or lower (if $A_{r_{i+1},c_{i+1}-1}$ ends with $\da$), so it is not a descent. 
Noting that there is no trailing $\ua$ in $A_{r,c}$ if $r=1$ or $c=\lambda_r$, it follows that 
$\des(T)=\sum_{r=2}^{k}\card{\Tu_{r}(T)}$.

Similarly, $i$ is an ascent of $T$ if and only if $A_{r_{i+1},c_{i+1}-1}$ has a trailing $\da$. Indeed, having a trailing $\da$ guarantees that entry $i$ is lower than $i+1$, so $i$ is an ascent. Conversely, if $A_{r_{i+1},c_{i+1}-1}$ does not have a trailing $\da$, then $i$ must have been placed in row $r_{i+1}$ (if $A_{r_{i+1},c_{i+1}-1}=\es$) or higher (if $A_{r_{i+1},c_{i+1}-1}$ ends with $\ua$), so it is not an ascent. 
Noting that there are no $\da$ in row $k$, 
and subtracting one for each $1\le r\le k-1$, since the trailing arrows in $A_{r,\lambda_r}$ (from the RB condition) do not contribute to ascents, we obtain $\asc(T)=\sum_{r=1}^{k-1}(\card{\Td_{r}(T)}-1)$.
\end{proof}

We can now prove that, on rectangular shapes, the map $\rev$ has the properties mentioned in Theorem~\ref{thm:rev}.

\begin{proof}[Proof of Theorem~\ref{thm:rev}]
Let $\rev:\SYT(n^k)\to\SYT(n^k)$ be the map from Definition~\ref{def:rev}, which is an involution by Lemma~\ref{lem:involution_rev}.

Applying $\rev$ to $T\in\SYT(n^k)$ turns trailing $\ua$ of its arrow configuration into leading $\ua$, and vice versa. Thus, we have
$\Lu_r(\rev(T))=\Tu_r(T)$ and $\Tu_r(\rev(T))=\Lu_r(T)$ for every $1\le r\le k$. Using both parts of Lemma~\ref{lem:desL},
$$\asc(\rev(T))+k-1=\sum_{r=2}^{k}\card{\Lu_{r}(\rev(T))}=\sum_{r=2}^{k}\card{\Tu_{r}(T)}=\des(T),$$
and similarly $\asc(T)+k-1=\des(\rev(T))$.
\end{proof}

For example, if $T$ is the tableau on the top left of Figure~\ref{fig:bij}, we have $\des(\rev(T))=15=\asc(T)+4$ and $\asc(\rev(T))=9=\des(T)-4$.

Next, we describe how the sets 
$\Ld_{r}(T)$ and $\Tu_{r+1}(T)$ are changed by the involution $\bijd_r$ from Definition~\ref{def:bijd} in the case of rectangular shapes. We denote the complement of a subset $S\subseteq\{0,1,\dots,n\}$ by $\ol{S}=\{0,1,\dots,n\}\setminus S$.

\begin{lemma}\label{lem:LT}
Let $T\in\SYT(n^k)$ and $\wh{T}=\bijd_r(T)$, where $1\le r\le k-1$. Then
$$\Ld_r(\wh{T})=\ol{\Tu_{r+1}(T)},\qquad \Tu_{r+1}(\wh{T})=\ol{\Ld_{r}(T)},$$
and $\Ld_s(\wh{T})=\Ld_s(T)$, $\Tu_{s+1}(\wh{T})=\Tu_{s+1}(T)$ for all $s\neq r$.
\end{lemma}

\begin{proof}
Let $\{A_{r,c}\}$ and $\{\wh{A}_{r,c}\}$ be the arrow encodings of $T$ and $\wh{T}$, respectively.
By Definition~\ref{def:bijd}, $\wh{A}_{r,c}$ has a leading $\da$ if and only if $A_{r+1,c}$ does not have a trailing $\ua$,
for every $0\le c\le n$, proving that $\Ld_r(\wh{T})=\ol{\Tu_{r+1}(T)}$. Similarly, $\wh{A}_{r+1,c}$ has a trailing $\ua$ if and only if $A_{r,c}$ does not have a leading $\da$, proving that $\Tu_{r+1}(\wh{T})=\ol{\Ld_{r}(T)}$.
 Finally, since the map $\bijd_r$ can only change leading $\da$ in row $r$ and trailing $\ua$ in row $r+1$, it preserves the sets $\Ld_s$ and $\Tu_{s+1}$ for all $s\neq r$.
\end{proof}

\begin{definition}\label{def:negswap} 
We will say that the map $\bijd_r$ {\em neg-swaps} leading $\da$ of $A_{r,c}$ with trailing $\ua$ of $A_{r+1,c}$, to refer to the behavior explained in the above proof.
\end{definition}

We can now prove that $\bijd$ has the desired behavior with respect to descents.

\begin{proof}[Proof of Theorem~\ref{thm:bijd}]
Let $\bijd:\SYT(n^k)\to\SYT(n^k)$ be the map from  Definition~\ref{def:bijd}, which is an involution by Lemma~\ref{lem:involution}.

By Lemma~\ref{lem:LT} applied to each of the $k-1$ maps in the composition $\bijd=\bijd_1\circ\bijd_2\circ\dots\circ\bijd_{k-1}$, we have
$$\Ld_r(\bijd(T))=\ol{\Tu_{r+1}(T)},\qquad \Tu_{r+1}(\bijd(T))=\ol{\Ld_{r}(T)}$$
for all $1\le r\le k-1$. Therefore, using the two equalities in the first part of Lemma~\ref{lem:desL},
$$\des(\bijd(T))=\sum_{r=1}^{k-1}\card{\Ld_{r}(\bijd(T))}=\sum_{r=1}^{k-1}\card{\ol{\Tu_{r+1}(T)}}=\sum_{r=1}^{k-1}\left(n+1-\card{\Tu_{r+1}(T)}\right)=(k-1)(n+1)-\des(T).\qedhere$$
\end{proof}

As an example, if $T$ is the tableau on the top left of Figure~\ref{fig:bij}, then $\des(T)=13$ and $\des(\bijd(T))=15$, so $\des(T)+\des(\bijd(T))=28=(5-1)(6+1)$.

On rectangular shapes, equation~\eqref{eq:bijad} allows us to use the properties of $\rev$ and $\bijd$ to prove that $\bija$ has the desired behavior.

\begin{proof}[Proof of Theorem~\ref{thm:bija}]
Let $\bija:\SYT(n^k)\to\SYT(n^k)$ be the map from Definition~\ref{def:bija}, which is an involution by Lemma~\ref{lem:involution}. 
Using the fact that $\bija=\rev\circ\bijd\circ\rev$ and applying Theorems \ref{thm:rev}, \ref{thm:bijd}, and~\ref{thm:rev} again, we have
$$\asc(\rev(\bijd(\rev(T))))=\des(\bijd(\rev(T)))-(k-1)=(k-1)n-\des(\rev(T))=(k-1)(n-1)-\asc(T).\qedhere$$
\end{proof}

Alternatively, we can prove Theorem~\ref{thm:bija} without referring to $\bijd$ or $\rev$, by instead using the following analogue of Lemma~\ref{lem:LT} for $\bija_r$, which has a very similar proof.

\begin{lemma}\label{lem:LT-bija}
Let $T\in\SYT(n^k)$ and $\wh{T}=\bija_r(T)$, where $1\le r\le k-1$. Then
$$\Td_r(\wh{T})=\ol{\Lu_{r+1}(T)},\qquad \Lu_{r+1}(\wh{T})=\ol{\Td_{r}(T)},$$
and $\Td_s(\wh{T})=\Td_s(T)$, $\Lu_{s+1}(\wh{T})=\Lu_{s+1}(T)$ for all $s\neq r$.
\end{lemma}

For the case of truncated staircases, in order to prove Theorem~\ref{thm:staircase}, we first need to slightly modify this lemma.

\begin{lemma}\label{lem:LTstaircase}
Let $\lambda=(n,n-1,\dots,n-k+1)$, where $1\le k\le n$.
Let $T\in\SYT(\lambda)$ and $\wh{T}=\bija_r(T)$, where $1\le r\le k-1$. Then
\begin{align}\label{eq:TdLu} 
\Td_r(\wh{T})&=\{0,1,\dots,n-r+1\}\setminus\Lu_{r+1}(T),\\ \label{eq:LuTd} 
\Lu_{r+1}(\wh{T})&=\{0,1,\dots,n-r+1\}\setminus\Td_{r}(T),
\end{align}
and $\Td_s(\wh{T})=\Td_s(T)$, $\Lu_{s+1}(\wh{T})=\Lu_{s+1}(T)$ for all $s\neq r$.
\end{lemma}

\begin{proof}
Let $\{A_{r,c}\}$ and $\{\wh{A}_{r,c}\}$ be the arrow encodings of $T$ and $\wh{T}$, respectively.
By Definition~\ref{def:bija_distinct},  for every $0\le c\le \lambda_{r+1}=n-r$, the sequence
$\wh{A}_{r,c}$ has a trailing $\da$ if and only if $A_{r+1,c}$ does not have a leading $\ua$.
Equation~\eqref{eq:TdLu} now follows, noting that $\lambda_r=n-r+1\in \Td_r(\wh{T})$ because of the RB condition from Definition~\ref{def:valid}, but $n-r+1\notin \Lu_{r+1}(T)$ because $n-r+1>\lambda_{r+1}=n-r$.

Similarly, $\wh{A}_{r+1,c}$ has a leading $\ua$ if and only if $A_{r,c}$ does not have a trailing $\da$.
Together with the fact that $n-r+1\notin \Lu_{r+1}(\wh{T})$ but $n-r+1\in \Td_r(T)$, we deduce equation~\eqref{eq:LuTd}. 

The sets $\Ld_s$ and $\Tu_{s+1}$ for $s\neq r$ are preserved because $\bija_r$ can only change leading $\da$ in row $r$ and trailing $\ua$ in row $r+1$.
\end{proof}

\begin{proof}[Proof of Theorem~\ref{thm:staircase}]
Let $\bija:\SYT(\lambda)\to\SYT(\lambda)$ be the map from Definition~\ref{def:bija_distinct}, where $\lambda=(n,n-1,\dots,n-k+1)$. We showed in Lemma~\ref{lem:involution_distinct} that $\bija$ is an involution. 

By Lemma~\ref{lem:LTstaircase} applied to each of the $k-1$ maps in the composition $\bija=\bija_1\circ\bija_2\circ\dots\circ\bija_{k-1}$, we have
$$\Td_r(\bija(T))=\{0,1,\dots,n-r+1\}\setminus\Lu_{r+1}(T),\qquad \Lu_{r+1}(\bija(T))=\{0,1,\dots,n-r+1\}\setminus\Td_{r}(T)$$
for all $1\le r\le k-1$. Using the two equalities in the second part of Lemma~\ref{lem:desL},
\begin{align*} \asc(\bija(T))+k-1&=\sum_{r=1}^{k-1}\card{\Td_{r}(\bija(T))}=\sum_{r=1}^{k-1}\left(n-r+2-\card{\Lu_{r+1}(T)}\right)\\
&=\frac{(2n-k+4)(k-1)}{2}-\left(\asc(T)+k-1\right),\end{align*}
from where equation~\eqref{eq:bij_asc} follows.
\end{proof}

Note that for the staircase $\lambda=(n,n-1,\dots,1)$, equation~\eqref{eq:bij_asc} becomes simply $$\asc(T)+\asc(\bija(T))=\binom{n}{2}.$$
As an example of this property of $\bija$ on staircases, if $T\in\SYT(5,4,3,2,1)$ is the tableau on the left of Figure~\ref{fig:bija-stair}, then $\asc(T)=6$ and $\asc(\bija(T))=4$, so $\asc(T)+\asc(\bijd(T))=10=\binom{5}{2}$.

\section{Other properties of the bijections}\label{sec:properties}

In this section we explore some other properties of our bijections for rectangular shapes from Section~\ref{sec:bij_rect}. 

\subsection{Symmetries of $\bijd$}\label{sec:symmetries}

Let us show that the involution $\bijd:\SYT(n^k)\to\SYT(n^k)$ from Definition~\ref{def:bijd}, in addition to proving Theorem~\ref{thm:bijd}, has some other interesting properties.

First, let us argue from the definition that $\bijd$ commutes with rotation (defined in Section~\ref{sec:basic}), that is, $\rot(\bijd(T))=\bijd(\rot(T))$ for all $T\in\SYT(n^k)$. Indeed, for $1\le r\le k-1$, leading $\da$ in row $r$ and trailing $\ua$ in row $r+1$ of $T$ become trailing $\ua$ in row $k+1-r$ and leading $\da$ in row $k-r$ of $\rot(T)$, respectively. Thus, applying $\bijd_r$ to $T$ has the same effect as applying $\bijd_{k-r}$ to $\rot(T)$, that is, $\rot(\bijd_r(T))=\bijd_{k-r}(\rot(T))$. It follows that applying $\bijd=\bijd_1\circ\bijd_2\circ\dots\circ\bijd_{k-1}$ to $T$ and then rotating the resulting tableau is equivalent to first rotating $T$ and then applying $\bijd=\bijd_{k-1}\circ\dots\circ\bijd_2\circ\bijd_1$ to the rotated tableau.

A similar argument shows that the involution $\bija$ from Definition~\ref{def:bija} commutes with rotation as well, that is, $\rot(\bija(T))=\bija(\rot(T))$ for all $T\in\SYT(n^k)$.

A much less obvious property is that $\bijd$ also commutes with conjugation. This may a priori be surprising considering that our definition of $\bijd$ treats rows and columns very differently.
For comparison, neither of the involutions $\bija$ and $\rev$ commutes with conjugation in general.

\begin{theorem}\label{thm:bijd-conj}
For all $T\in\SYT(n^k)$, we have $\bijd(T)'=\bijd(T')$.
\end{theorem}

\begin{proof}
Let $\{A_{r,c}\}_{1\le r\le k,0\le c\le n}$ be the arrow encoding of $T$, and let $\{A'_{c,r}\}_{1\le c\le n,0\le r\le k}$ be the arrow encoding of its conjugate $T'$. 
Let $1\le r\le k$ and $1\le c\le n$. We will show that $A_{r,c}$ has a leading $\da$ if and only if $A'_{c,r}$ does not have a leading $\da$. 
Let $i$ be the entry in cell $(r,c)$ of $T$, or equivalently, in cell $(c,r)$ of $T'$. 
As in the proof of Lemma~\ref{lem:desL}, we have the following equivalent statements:
\begin{center}
$A_{r,c}$ has a leading $\da\  \Leftrightarrow \ i\in\Des(T) \ \Leftrightarrow \ i\notin\Des(T')  \ \Leftrightarrow \ A'_{c,r}$ does not have a leading $\da$.
\end{center}
In other words, exactly one of $A_{r,c}$ and $A'_{c,r}$ has a leading $\da$.

Now let $0\le r\le k-1$ and $0\le c\le n-1$. We will show that $A_{r+1,c}$ has a trailing $\ua$ if and only if $A'_{c+1,r}$ does not have a trailing $\ua$. 
Let $j+1$ be the entry in cell $(r+1,c+1)$ of $T$, or equivalently, in cell $(c+1,r+1)$ of $T'$. 
As in the proof of Lemma~\ref{lem:desL}, we have the following equivalent statements:
\begin{center}
$A_{r+1,c}$ has a trailing $\ua\  \Leftrightarrow \ j\in\Des(T) \ \Leftrightarrow \ j\notin\Des(T')  \ \Leftrightarrow \ A'_{c+1,r}$ does not have a trailing $\ua$.
\end{center}
In other words, exactly one of $A_{r+1,c}$ and $A'_{c+1,r}$ has a trailing $\ua$.

Using the terminology from Definition~\ref{def:negswap}, the map $\bijd$ applied to $T$ neg-swaps leading $\da$ of $A_{r,c}$ with trailing $\ua$ of $A_{r+1,c}$
for each $1\le r\le k-1$ and $1\le c\le n-1$. By the above equivalences, this has the same effect as neg-swapping leading $\da$ of $A'_{c,r}$ with trailing $\ua$ of $A'_{c+1,r}$, which is precisely what the map $\bijd$ does when applied to $T'$. Thus, applying $\bijd$ to $T$ and then conjugating is equivalent to applying $\bijd$ to $T'$.
\end{proof}

\subsection{Another symmetry of the generalized Narayana numbers}

Next we show that our involutions from Section~\ref{sec:bij} can be used to provide a bijective proof of another property of the generalized Narayana numbers.

\begin{proposition}[\cite{sulanke_generalizing_2004}]\label{prop:n<->k} 
For $0\le h\le (k-1)(n-1)$, we have
$$N(k,n,h)=N(n,k,h).$$
\end{proposition}

Note that this symmetry is immediate from the expression in Theorem~\ref{thm:Narayana-formula}. It is also equivalent to \cite[Prop.~6]{sulanke_generalizing_2004}, where it is formulated in terms high descents, but no bijective proof is given. 
For any $T\in\SYT(\lambda)$, a descent $i\in\Des(T)$ is called a {\em high descent} if $i+1$ is strictly to the left of $i$. In the case $T\in\SYT(n^2)$, the bijection $\TD$ from equation~\eqref{eq:TD} takes high descents of $T$ to {\em high peaks} of the corresponding Dyck path, which are adjacent pairs $\uu\dd$ that do not touch the $x$-axis. 
Let $\HDes(T)$ denote the set of high descents of $T$, and let $\hdes(T)=\card{\HDes(T)}$. Note that $\HDes(T)=\Asc(T')$ for any $T\in\SYT(\lambda)$, and so, by equation~\eqref{eq:Nasc},
$$N(n,k,h)=|\{T\in\SYT(n^k):\hdes(T)=h\}|.$$
Using the interpretation of $N(k,n,h)$ in terms of ascents, the following result provides a bijective proof of Proposition~\ref{prop:n<->k}.

\begin{proposition}\label{prop:rowmotion}
There exists a bijection $\rowmotion:\SYT(n^k)\to\SYT(n^k)$ such that, for all $T\in\SYT(n^k)$,
$$\asc(T)=\hdes(\rowmotion(T)).$$
\end{proposition}

\begin{proof}
For $T\in\SYT(n^k)$, define $\rowmotion(T)=\rev(\bijd(\rev(T))')'$.
By Theorems~\ref{thm:rev} and~\ref{thm:bijd}, and equation~\eqref{eq:desT'},
$$\asc(T)=\des(\rev(T))-(k-1)=(k-1)n-\des(\bijd(\rev(T)))=\des(\bijd(\rev(T))')-(n-1).$$
Since $\bijd(\rev(T))'\in\SYT(k^n)$, Theorem~\ref{thm:rev} (with $k$ and $n$ switched), followed by conjugation, gives
$$\des(\bijd(\rev(T))')-(n-1)=\asc(\rev(\bijd(\rev(T))'))=\hdes(\rev(\bijd(\rev(T))')')=\hdes(\rowmotion(T)).\qedhere$$
\end{proof}

With the interpretation of the generalized Narayana numbers given by equation~\eqref{eq:Ndes}, Proposition~\ref{prop:n<->k} has a slightly more direct proof by using the bijection $T\mapsto \bijd(T)'$. Indeed, $N(k,n,h)$ is the number of tableaux in $\SYT(n^k)$ with $h+k-1$ descents. Applying $\bijd$, this is also the number of tableaux in $\SYT(n^k)$ with $(k-1)n-h$ descents, by Theorem~\ref{thm:bijd}. And applying conjugation, this equals the number of tableaux in $\SYT(k^n)$ with $h+n-1$ descents, which equals $N(n,k,h)$ again by equation~\eqref{eq:Ndes}.

As we will discuss in Section~\ref{sec:posets}, Proposition~\ref{prop:n<->k} follows from the fact that any two natural labelings of a poset give rise to the same descent polynomial~\cite{stanley_enumerative_2012}. In fact, the theory of $P$-partitions implies that the statistics $\asc$ and $\hdes$ (and, more generally, the set-valued statistics $\Asc$ and $\HDes$) are equidistributed not only on $\SYT(n^k)$, but on $\SYT(\lambda)$ for any shape $\lambda$. 
Unfortunately, this argument is not bijective, but we will give a proof by inclusion-exclusion in Section~\ref{sec:independence}.

\subsection{Recovering the classical Lalanne--Kreweras involution}\label{sec:LK}

Let us show that, in the special case of rectangular shapes with $k=2$ rows, the maps $\bijd$ and $\bija$ from Definitions~\ref{def:bijd} and~\ref{def:bija} are equivalent to the Lalanne--Kreweras involution for Dyck paths~\cite{lalanne_involution_1992,kreweras_sur_1970}, which we denote by $\LK$.

To describe this involution, following~\cite{elizalde_descents_2024}, it will be convenient to draw Dyck paths in $\D_n$ as lattice paths from $(0,0)$ to $(n,n)$ with steps $\nn=(0,1)$ and $\ee=(1,0)$ playing the roles of $\uu$ and~$\dd$.  If $\nn$ (respectively $\ee$) is the one playing the role of $\uu$, the resulting path stays weakly above (resp.\ below) the diagonal, and peaks correspond to corners of the form $\nn\ee$ (resp.~$\ee\nn$). 
Given $D\in\D_n$, drawn as a path from $(0,0)$ to $(n,n)$ staying weakly above the diagonal, suppose that the coordinates of its peaks are $(x_1,y_1),(x_2,y_2),\dots,(x_{h+1},y_{h+1})$, where $0=x_1<x_2<\dots<x_{h+1}<n$ and $0<y_1<y_2<\dots<y_{h+1}=n$), and write
\begin{align*}
\{0,1,\dots,n\}\setminus\{x_1,x_2,\dots,x_{h+1}\}&=\{x'_1,x'_2,\dots,x'_{n-h}\},\\
\{0,1,\dots,n\}\setminus\{y_1,y_2,\dots,y_{h+1}\}&=\{y'_1,y'_2,\dots,y'_{n-h}\},
\end{align*}
where $0<x'_1<x'_2<\dots<x'_{n-h}=n$ and $0=y'_1<y'_2<\dots<y'_{n-h}<n$. Then $\LK(D)$
is the Dyck path that, when drawn as a path from $(0,0)$ to $(n,n)$ staying weakly below the diagonal, has its peaks at coordinates $(x'_1,y'_1),(x'_2,y'_2),\dots,(x'_{n-h},y'_{n-h})$.
See the right of Figure~\ref{fig:k=2} for an example.

\begin{figure}[htb]
\centering
\begin{tikzpicture}[scale=.58]

\begin{scope}[shift={(-12,4.5)},red,scale=1.25]
\drawtab{7}{2}{{{1,2,4,6,8,9,13},{3,5,7,10,11,12,14}}}
\drawarr{2}{{{,,\da,\da,\da,,\da,\da},{\ua,\ua,\ua,\ua,,,\ua,}}}
\draw[<->,black] (3.5,-.4)--node[left]{$\bijd$}(3.5,-1.1);
\end{scope}
\begin{scope}[shift={(-12,0)},teal,scale=1.3]
\drawtab{7}{2}{{{1,2,3,4,6,11,12},{5,7,8,9,10,13,14}}}
\drawarr{2}{{{,,,,\da,\da,,\da},{\ua,\ua,,,,\ua,,}}}
\end{scope}

\doublegrid{7}   
 \foreach [count=\i] \x in {4,5,7} {
    	\draw[teal,->] (\x,-.6)--(\x,-.2);
	 \draw[teal] (\x,-1) node[scale=.9] {$x'_\i$};
 }
 \foreach [count=\i] \y in {0,1,5} {
    	\draw[teal,->] (7.6,\y)--(7.2,\y);
	\draw[teal] (8,\y) node[scale=.9] {$y'_\i$};
 }
 \draw[very thick,teal](0,0)  \e\e\e\e\n\e\n\n\n\n\e\e\n\n;
 \foreach [count=\i] \x in {0,1,2,3,6} {goal
    	\draw[red,->] (\x,-.6)--(\x,-.2);
	\draw[red] (\x,-1) node[scale=.9] {$x_\i$};
 }
 \foreach [count=\i] \y in {2,3,4,6,7} {
    	\draw[red,->] (7.6,\y)--(7.2,\y);
	\draw[red] (8,\y) node[scale=.9] {$y_\i$};	
 }
\draw[very thick,red](0,0)  \n\n\e\n\e\n\e\n\n\e\e\e\n\e;
\end{tikzpicture}
\caption{The involutions $\bijd$ and $\bija$ on $\SYT(n^2)$ are equivalent to the Lalanne--Kreweras involution on $\D_n$.}
\label{fig:k=2}
\end{figure}

When $k=2$ in Definitions~\ref{def:bijd} and~\ref{def:bija}, we have $\bijd=\bijd_1$ and $\bija=\bija_1$. In this special case, these maps coincide with each other because the first row of the arrow encoding of a tableau $T\in\SYT(n^2)$ consists of $\da$ and $\es$ only, while the second row consists of $\ua$ and $\es$ only, so all arrows are both leading and trailing. Let $\wh{T}=\bijd_1(T)=\bija_1(T)$, and let $\{A_{r,c}\}$ and $\{\wh{A}_{r,c}\}$ be the arrow encodings of $T$ and $\wh{T}$, respectively.
Then $\wh{A}_{1,c}=\da$ if and only if $A_{2,c}=\es$, and $\wh{A}_{2,c}=\ua$ if and only if $A_{1,c}=\es$ since, as described in Definition~\ref{def:negswap}, the map $\bijd_1$ neg-swaps $\da$ in the first row with $\ua$ in the second row.

Suppose that $\TD(T)=D$, where  $\TD$ is the bijection from equation~\eqref{eq:TD}. A property of this bijection is that, if the Dyck path $D\in\D_n$, when drawn as a path staying above the diagonal, has peaks at coordinates $(x_1,y_1),(x_2,y_2),\dots,(x_{h+1},y_{h+1})$, then $A_{1,c}=\da$ if and only if $c\in\{x_1,x_2,\dots,x_{h+1}\}$, and $A_{2,c}=\ua$ if and only if $c\in\{y_1,y_2,\dots,y_{h+1}\}$.
After applying $\bijd_1$, we have 
\begin{align*} \wh{A}_{1,c}=\da & \ \Leftrightarrow\ A_{2,c}=\es \ \Leftrightarrow\ c\in\{0,1,\dots,n\}\setminus\{y_1,y_2,\dots,y_{h+1}\}=\{x'_1,x'_2,\dots,x'_{n-h}\},\\
\wh{A}_{2,c}=\ua & \ \Leftrightarrow\  A_{1,c}=\es \ \Leftrightarrow\  c\in\{0,1,\dots,n\}\setminus\{x_1,x_2,\dots,x_{h+1}\}=\{y'_1,y'_2,\dots,y'_{n-h}\}.\end{align*}
Thus, when drawing the Dyck path $\TD(\wh{T})$ as a path staying below the diagonal, the coordinates of its peaks are $(x'_1,y'_1),(x'_2,y'_2),\dots,(x'_{n-h},y'_{n-h})$. It follows that this path is precisely $\LK(D)$. See Figure~\ref{fig:k=2} for an example.

Finally, we remark that, for rectangular tableaux with $n=2$ columns, the map $\bijd$ from Definition~\ref{def:bijd} is also equivalent to the Lalanne--Kreweras involution, by using the bijection $T\mapsto\TD(T')$ between $\SYT(2^k)$ and $\D_k$, and then the property, proved in Theorem~\ref{thm:bijd-conj}, that $\bijd(T)=\bijd(T')'$.

\section{Refined descent statistics on rectangular tableaux}\label{sec:refined}

\subsection{The statistics $\b_{r,s}$}

In this section we define statistics that refine $\asc$ and $\des$ by keeping track not only of the relative position of the entries $i$ and $i+1$, but also of which rows they lie in. We then study how the maps $\bijd$, $\bija$ and $\beta$ on $\SYT(n^k)$ behave with respect to these refined statistics.

For $T\in\SYT(\lambda)$ and $1\le r,s\le k$, let 
$$\b_{r,s}(T)=|\{i:\row(i)=r\text{ and }\row(i+1)=s\}|,$$
where $1\le i\le N-1$. We will write $\b_{r,s}$ instead of $\b_{r,s}(T)$ when it creates no confusion. If $\row(i)=r$ and $\row(i+1)=s$, we will say that $i$ is a {\em bounce} from row $r$ to row $s$. 
Let $B(T)$ denote the $k\times k$ matrix with entries $b_{r,s}$, which records the number of bounces between the rows of~$T$.

The argument in the proof of Lemma~\ref{lem:desL} shows that $\card{\Ld_{r}(T)}$ equals the number of $i\in\Des(T)$ such that $\row(i)=r$, 
$\card{\Lu_{r}(T)}-1$ is the number of $i\in\Asc(T)$ such that $\row(i)=r$, 
$\card{\Tu_{r}(T)}$ is the number of $i\in\Des(T)$ such that $\row(i+1)=r$,
and 
$\card{\Td_{r}(T)}-1$ is the number of $i\in\Asc(T)$ such that $\row(i+1)=r$.  It follows that, for $1\le r\le k-1$,
\begin{align}
\label{eq:Ld-b}
\card{\Ld_{r}(T)}&=\b_{r,r+1}+\b_{r,r+2}+\dots+\b_{r,k}, & \quad
\card{\Tu_{r+1}(T)}&=\b_{1,r+1}+\b_{2,r+1}+\dots+\b_{r,r+1},\\
\label{eq:Lu-b}
\card{\Td_{r}(T)}&=\b_{r+1,r}+\b_{r+2,r}+\dots+\b_{k,r}+1, &
\card{\Lu_{r+1}(T)}&=\b_{r+1,1}+\b_{r+1,2}+\dots+\b_{r+1,r}+1.
\end{align}

In the rest of this section we will focus on the case of rectangular shapes $\lambda=(n^k)$. 
In this case, the largest entry of $T\in\SYT(n^k)$ must be in row $k$.
It follows that 
\begin{align}
\label{eq:sumb-row} \b_{r,1}+\b_{r,2}+\dots+\b_{r,k}&=\begin{cases} n & \text{if }1\le r\le k-1, \\ n-1 & \text{if }r=k,\end{cases}\\
\b_{1,r}+\b_{2,r}+\dots+\b_{k,r}&=\begin{cases} n & \text{if }2\le r\le k, \\ n-1 & \text{if }r=1.\end{cases}
\label{eq:sumb-col}
\end{align}

The following is a refinement of Lemma~\ref{lem:LT}.

\begin{lemma}\label{lem:brs}
For $1\le r\le k-1$, the involution $\bijd_r$ on $\SYT(n^k)$ swaps the statistics
\begin{align}
\b_{r,1}+\b_{r,2}+\dots+\b_{r,r}&\swap \b_{1,r+1}+\b_{2,r+1}+\dots+\b_{r,r+1}-1, \label{eq:swap1}\\
\b_{r,r+1}+\b_{r,r+2}+\dots+\b_{r,k}-1&\swap \b_{r+1,r+1}+\b_{r+2,r+1}+\dots+\b_{k,r+1}, \label{eq:swap2}
\end{align} and it preserves the statistics $\b_{j,\ell}$ for $1\le j\le r-1$ and $1\le \ell\le r$, and 
the statistics $\b_{j,\ell}$ for $r+1\le j\le k$ and $r+2\le \ell\le k$.
\end{lemma}

\begin{proof}
Using that $\card{\ol{\Ld_{r}(T)}}=n+1-\card{\Ld_{r}(T)}$ and $\card{\ol{\Tu_{r+1}(T)}}=n+1-\card{\Tu_{r+1}(T)}$, and that both the $r$th row and the $(r+1)$st column of $B(T)$ sum to $n$, equation~\eqref{eq:Ld-b} implies that
$\card{\ol{\Ld_{r}(T)}}=\b_{r,1}+\b_{r,2}+\dots+\b_{r,r}+1$ and $\card{\ol{\Tu_{r+1}(T)}}=\b_{r+1,r+1}+\b_{r+2,r+1}+\dots+\b_{k,r+1}+1$.
Thus, by Lemma~\ref{lem:LT}, $\bijd_r$ swaps the statistics in~\eqref{eq:swap1}, or equivalently the ones in~\eqref{eq:swap2}.

To see that $\bijd_r$ preserves the statistics $\b_{j,\ell}$ for $1\le j,\ell\le r-1$, note that $\bijd_r$ does not change any of the arrows in the first $r-1$ rows of the arrow encoding, so the algorithm in the proof of Lemma~\ref{lem:encoding} follows these arrows in the same order in $T$ as in $\bijd_r(T)$, hence the bounces between the first $r-1$ rows are the same. A similar argument shows that 
$\bijd_r$ preserves the statistics $\b_{j,\ell}$ for $r+2\le j,\ell\le k$

To see that $\bijd_r$ also preserves $\b_{j,r}$ for $1\le j\le r-1$, note that bounces into row $r$ coming from higher rows are not affected by the presence or absence of leading $\da$ in row $r$, which are the only arrows that $\bijd_r$ could change in that row. The fact that $\bijd_r$ changes no other arrows in row $r$, and no arrows in the first $r-1$ rows, guarantees that the bounces from any of the first $r-1$ rows into row $r$ are the same in $T$ as in $\bijd_r(T)$.
By a similar reasoning, $\bijd_r$ also preserves $\b_{r+1,\ell}$ for $r+2\le \ell \le k$.

In fact, the above argument shows that in all these cases, not only the number $\b_{j,\ell}$ of bounces from row $j$ to row $\ell$ is preserved, but also the specific cells in rows $j$ and $\ell$ between which these bounces occur.
\end{proof}

As an example, the bounce matrices of the tableaux in Figure~\ref{fig:bijd_r} are
$$B(T)=
\begin{tikzpicture}[baseline]
\matrix(m)[
    matrix of math nodes, 
    row sep=2pt,
    column sep=2pt,
    left delimiter={[},
    right delimiter={]}, 
    ]{
     |[draw, rectangle]|2 & |[draw, rectangle]|2 & |[draw, rectangle]|1 & 1 & 0 \\ 
    |[draw, rectangle]|1 & |[draw, rectangle]|1 & |[draw, rectangle]|1 & 3 & 0 \\ 
    1 & 3 & 0  & 1 & 1 \\ 
    0 & 0 & 3 & 0 & |[draw, rectangle]|3\\
    1 & 0 & 1 & 1 & |[draw, rectangle]|2\\
    };
    \draw[violet,rounded corners] (m-3-4.north west) rectangle (m-3-5.south east);
    \draw[darkgreen,rounded corners] (m-4-4.north west) rectangle (m-5-4.south east);
    \draw[blue,rounded corners] (m-3-1.north west) rectangle (m-3-3.south east);
    \draw[red,rounded corners] (m-1-4.north west) rectangle (m-3-4.south east);
    \draw[red] (m-1-4.north east) node[scale=.7,above,yshift=-1mm] {$-1$};
    \draw[violet] (m-3-5.north east) node[scale=.7,above,yshift=-1.5mm] {$-1$};
\end{tikzpicture}\qquad\text{and}\qquad
B(\bijd_3(T))=
\begin{tikzpicture}[baseline]
\matrix(m)[
    matrix of math nodes, 
    row sep=2pt,
    column sep=2pt,
    left delimiter={[},
    right delimiter={]}, 
    ]{
     |[draw, rectangle]|2 & |[draw, rectangle]|2 & |[draw, rectangle]|1 & 1 & 0 \\ 
    |[draw, rectangle]|1 & |[draw, rectangle]|1 & |[draw, rectangle]|1 & 2 & 1 \\ 
    1 & 2 & 1  & 2 & 0 \\ 
    0 & 1 & 1 & 1 & |[draw, rectangle]|3\\
    1 & 0 & 2 & 0 & |[draw, rectangle]|2\\
    };
   \draw[darkgreen,rounded corners] (m-3-4.north west) rectangle (m-3-5.south east);
    \draw[violet,rounded corners] (m-4-4.north west) rectangle (m-5-4.south east);
    \draw[red,rounded corners] (m-3-1.north west) rectangle (m-3-3.south east);
    \draw[blue,rounded corners] (m-1-4.north west) rectangle (m-3-4.south east);
        \draw[blue] (m-1-4.north east) node[scale=.7,above,yshift=-1mm] {$-1$};
    \draw[darkgreen] (m-3-5.north east) node[scale=.7,above,yshift=-1.5mm] {$-1$};
\end{tikzpicture}.
$$
The black squares indicate the entries preserved by $\bijd_3$ according to Lemma~\ref{lem:brs}, and the colored boxes with rounded corners indicate entries whose sums (minus one in some cases) are swapped.

\begin{remark}\label{rem:bijd1}
When $r=1$, Lemma~\ref{lem:brs} implies that $\bijd_1$ in fact preserves all the statistics $\b_{j,\ell}$ for $1\le j\le k$ and $3\le \ell\le k$. This is because $\b_{1,\ell}$ can be expressed in terms of preserved statistics as
$\b_{1,\ell}=n-\b_{2,\ell}-\dots-\b_{k,\ell}$, by equation~\eqref{eq:sumb-col}.
Similarly, when $r=k-1$, the map $\bijd_{k-1}$ preserves the statistics $\b_{j,\ell}$ for $1\le j\le k-2$ and $1\le \ell\le k$.
\end{remark}

For the map $\bija$, following refinement of Lemma~\ref{lem:LT-bija} holds. 
The proof is analogous to that of Lemma~\ref{lem:brs}, although the cases $r=1$ and $r=k-1$ have to be treated slightly different because of equations~\eqref{eq:sumb-row} and~\eqref{eq:sumb-col}. We use the Kronecker delta notation: $\delta_{i,j}=1$ if $i=j$, and $\delta_{i,j}=0$ otherwise.

\begin{lemma}\label{lem:brs-bija}
For $1\le r\le k-1$, the involution $\bija_r$ on $\SYT(n^k)$ swaps the statistics
\begin{align*}
\b_{1,r}+\b_{2,r}+\dots+\b_{r,r}-1+\delta_{r,1}&\swap \b_{r+1,1}+\b_{r+1,2}+\dots+\b_{r+1,r}, \\
\b_{r+1,r}+\b_{r+2,r}+\dots+\b_{k,r}&\swap \b_{r+1,r+1}+\b_{r+1,r+2}+\dots+\b_{r+1,k}-1+\delta_{r+1,k},
\end{align*} and it preserves the statistics $\b_{j,\ell}$ for $1\le j\le r$ and $1\le \ell\le r-1$, and 
the statistics $\b_{j,\ell}$ for $r+2\le j\le k$ and $r+1\le \ell\le k$.
\end{lemma}

The new lemma describes the effect of the involution~$\rot$ from Section~\ref{sec:basic} on the defined descent statistics.

\begin{lemma}\label{lem:rot}
The rotation map $\rot$, when applied to $T\in\SYT(n^k)$, reflects the matrix $B(T)$ along the secondary diagonal, that is,
$\b_{r,s}(T)=\b_{k+1-s,k+1-r}(\rot(T))$ for all $1\le r,s \le k$.
\end{lemma}

\begin{proof}
Let $1\le i\le kn-1$, and suppose that $i$ is in row $r$ and $i+1$ is in row $s$ of $T$. Then $n+1-i$ is in row $k+1-r$ and $n-i$ is in row $k+1-s$ of $\rot(T)$. Thus, $i$ contributes to $\b_{r,s}(T)$ if and only if $n-i$ contributes to $\b_{k+1-s,k+1-r}(\rot(T))$.
\end{proof}

We can write the map from Definition~\ref{def:rev} as a composition $\rev=\rev_2\circ\rev_3\circ\dots\circ\rev_{k-1}$, where, for each $2\le r\le k-1$, $\rev_r$ reverses the sequences $A_{r,c}$ in row $r$ of the arrow encoding. 
The next lemma describes the behavior of the map $\rev_r$ with respect to the refined descent statistics.

\begin{lemma}\label{lem:revr}
For $2\le r\le k-1$, the map $\rev_r$ is an involution on $\SYT(n^k)$ that swaps the statistics
\begin{align}
\b_{r,1}+\b_{r,2}+\dots+\b_{r,r-1}&\swap \b_{1,r}+\b_{2,r}+\dots+\b_{r-1,r}-1, \label{eq:swap1rev}\\
\b_{r,r+1}+\b_{r,r+2}+\dots+\b_{r,k}-1&\swap \b_{r+1,r}+\b_{r+2,r}+\dots+\b_{k,r}. \label{eq:swap2rev}
\end{align} and preserves the statistics $\b_{j,\ell}$ for $1\le j,\ell\le r-1$, the statistics $\b_{j,\ell}$ for $r+1\le j,\ell\le k$, and the statistic $\b_{r,r}$.
\end{lemma}

\begin{proof}
Since reversing the sequences in any row of a valid arrow array produces another valid array, it is clear that $\rev_r$ is well defined, and that it is an involution.

Equation~\eqref{eq:swap1rev} follows from the fact that $\rev_r$ swaps $\card{\Lu_r(T)}$ with $\card{\Tu_r(T)}$, using equations~\eqref{eq:Ld-b} and~\eqref{eq:Lu-b}. Similarly, equation~\eqref{eq:swap2rev} follows from the fact that $\rev_r$ swaps $\card{\Ld_r(T)}$ with $\card{\Td_r(T)}$.

Additionally, $\rev_r$ preserves the statistics  $\b_{j,\ell}$ for $1\le j,\ell\le r-1$ because $\rev_r$ does not change the arrows in the first $r-1$ rows of the arrow encoding. It also preserves the statistics $\b_{j,\ell}$ for $r+1\le j,\ell\le k$ for similar reasons. Finally, $\b_{r,r}$ is the number of empty sequences $A_{r,c}=\es$ in row $r$ of the arrow encoding, which is not changed by $\rev_r$.
\end{proof}

\subsection{Three-row rectangular tableaux and a conjecture of Sulanke}

As we showed in Section~\ref{sec:LK}, the maps $\bijd=\bijd_1$ and $\bija=\bija_1$ on $\SYT(n^2)$ coincide with the classical Lalanne--Kreweras involution, whereas the map $\rev$ is simply the identity on this set. 
On $\SYT(n^3)$, these maps become more interesting. In this section we focus on the behavior of our bijections on $\SYT(n^3)$ with respect to the statistics $\b_{r,s}$, and we use them to prove a conjecture of Sulanke~\cite{sulanke_three_2005}.
We also conjecture the existence of additional  bijections with specific statistic-preserving properties.

The following lemma, illustrated in Figure~\ref{fig:k=3}, summarizes the behavior of the involutions $\bijd_1$, $\bijd_2$, $\bija_1$, $\bija_2$, $\rot$ and $\rev$ on the refined descents statistics when $k=3$. Recall that $\bija_r=\rev\circ\bijd_r\circ\rev$, by the argument above equation~\eqref{eq:bijad}, and that $\bijd_2=\rot\circ\bijd_1\circ\rot$ and $\bija_2=\rot\circ\bija_1\circ\rot$, by the argument at the beginning of Section~\ref{sec:symmetries}.

\begin{lemma}\label{lem:3preserved}
On $SYT(n^3)$,
\begin{enumerate}[label=(\alph*)]
\item $\bijd_1$ swaps $\b_{1,1}\swap \b_{1,2}-1$, and it preserves $\b_{1,3}$, $\b_{2,3}$ and $\b_{3,3}$;
\item $\bijd_2$ swaps $\b_{2,3}-1\swap \b_{3,3}$, and it preserves $\b_{1,1}$, $\b_{1,2}$ and $\b_{1,3}$;
\item $\bija_1$ swaps $\b_{1,1}\swap \b_{2,1}$, and it preserves $\b_{3,1}$, $\b_{3,2}$ and $\b_{3,3}$;
\item $\bija_2$ swaps $\b_{3,2}\swap \b_{3,3}$, and it preserves $\b_{1,1}$, $\b_{2,1}$ and $\b_{3,1}$;
\item $\rot$ swaps $\b_{1,1}\swap\b_{3,3}$, $\b_{1,2}\swap\b_{2,3}$ and $\b_{2,1}\swap\b_{3,2}$, and it preserves $\b_{1,3}$, $\b_{2,2}$ and $\b_{3,1}$;
\item $\rev$ swaps  $\b_{2,1}\swap\b_{1,2}-1$, $\b_{2,3}-1\swap\b_{3,2}$ and $\b_{1,3}\swap\b_{3,1}$, and it preserves $\b_{1,1}$, $\b_{2,2}$ and $\b_{3,3}$.
\end{enumerate}
\end{lemma}

\begin{figure}[htb]
$$\begin{array}{cccc}
\bijd_1 & \bijd_2 & \bija_1 & \bija_2\\ 
\begin{tikzpicture}[baseline]
\matrix(m)[
    matrix of math nodes, 
    row sep=6pt,
    column sep=6pt,
    left delimiter={[},
    right delimiter={]}, 
    ]{
     |[draw, rectangle,red, rounded corners]|\textcolor{black}{\b_{1,1}} &   |[draw, rectangle,red,rounded corners]|\textcolor{black}{\b_{1,2}}  & |[draw, rectangle]|\b_{1,3}  \\ 
     &  & |[draw, rectangle]|\b_{2,3}  \\ 
    &  & |[draw, rectangle]|\b_{3,3} \\
    };
    \draw[red] (m-1-2.north east) node[scale=.7,above,yshift=-1.5mm] {$-1$};
    \draw[<->,thick,red] (m-1-1.east)--(m-1-2.west);
\end{tikzpicture}
&
\begin{tikzpicture}[baseline]
\matrix(m)[
    matrix of math nodes, 
    row sep=6pt,
    column sep=6pt,
    left delimiter={[},
    right delimiter={]}, 
    ]{
     |[draw, rectangle]|\b_{1,1} & |[draw, rectangle]|\b_{1,2} & |[draw, rectangle]|\b_{1,3}  \\ 
   && |[draw, rectangle, red, rounded corners]|\textcolor{black}{\b_{2,3}} \\ 
  & & |[draw, rectangle, red, rounded corners]|\textcolor{black}{\b_{3,3}}\\
    };
    \draw[red] (m-2-3.north east) node[scale=.7,above,yshift=-2mm,xshift=1mm] {$-1$};
    \draw[<->,thick,red] (m-2-3.south)--(m-3-3.north);
\end{tikzpicture}
&
\begin{tikzpicture}[baseline]
\matrix(m)[
    matrix of math nodes, 
    row sep=6pt,
    column sep=6pt,
    left delimiter={[},
    right delimiter={]}, 
    ]{
     |[draw, rectangle,red, rounded corners]|\textcolor{black}{\b_{1,1}} &  &  \\ 
      |[draw, rectangle,red,rounded corners]|\textcolor{black}{\b_{2,1}}  &  &   \\ 
     |[draw, rectangle]|\b_{3,1} & |[draw, rectangle]|\b_{3,2} & |[draw, rectangle]|\b_{3,3} \\
    };
    \draw[<->,thick,red] (m-1-1.south)--(m-2-1.north);
\end{tikzpicture}
&
\begin{tikzpicture}[baseline]
\matrix(m)[
    matrix of math nodes, 
    row sep=6pt,
    column sep=6pt,
    left delimiter={[},
    right delimiter={]}, 
    ]{
     |[draw, rectangle]|\b_{1,1} & & \\ 
   |[draw, rectangle]|\b_{2,1}  & & \\ 
   |[draw, rectangle]|\b_{3,1}  &  |[draw, rectangle, red, rounded corners]|\textcolor{black}{\b_{3,2}} & |[draw, rectangle, red, rounded corners]|\textcolor{black}{\b_{3,3}}\\
    };
    \draw[<->,thick,red] (m-3-2.east)--(m-3-3.west);
\end{tikzpicture}
\medskip\\
& \rot &\rev & \\
&\begin{tikzpicture}[baseline]
\matrix(m)[
    matrix of math nodes, 
    row sep=6pt,
    column sep=6pt,
    left delimiter={[},
    right delimiter={]}, 
    ]{
     |[draw, rectangle, red, rounded corners]|\textcolor{black}{\b_{1,1}} & |[draw, rectangle, blue, rounded corners]|\textcolor{black}{\b_{1,2}} & |[draw, rectangle]|\b_{1,3}  \\ 
    |[draw, rectangle, darkgreen, rounded corners]|\textcolor{black}{\b_{2,1}}& |[draw, rectangle]|\b_{2,2} & |[draw, rectangle, blue, rounded corners]|\textcolor{black}{\b_{2,3}} \\ 
    |[draw, rectangle]|\b_{3,1} &  |[draw, rectangle, darkgreen, rounded corners]|\textcolor{black}{\b_{3,2}}& |[draw, rectangle,red, rounded corners]|\textcolor{black}{\b_{3,3}}\\
    };
    \draw[<->,thick,red] (m-1-1.south east)--(m-3-3.north west);
    \draw[<->,thick,blue] (m-1-2.south east)--(m-2-3.north west);   
    \draw[<->,thick,darkgreen] (m-2-1.south east)--(m-3-2.north west);   
\end{tikzpicture}
&
\begin{tikzpicture}[baseline]
\matrix(m)[
    matrix of math nodes, 
    row sep=6pt,
    column sep=6pt,
    left delimiter={[},
    right delimiter={]}, 
    ]{
   |[draw, rectangle]|\b_{1,1} & |[draw, rectangle, blue, rounded corners]|\textcolor{black}{\b_{1,2}} &   |[draw, rectangle, red, rounded corners]|\textcolor{black}{\b_{1,3}}  \\ 
    |[draw, rectangle, blue, rounded corners]|\textcolor{black}{\b_{2,1}}& |[draw, rectangle]|\b_{2,2} & |[draw, rectangle, darkgreen, rounded corners]|\textcolor{black}{\b_{2,3}} \\ 
    |[draw, rectangle,red, rounded corners]|\textcolor{black}{\b_{3,1}} &  |[draw, rectangle, darkgreen, rounded corners]|\textcolor{black}{\b_{3,2}}&  |[draw, rectangle]|\b_{3,3} \\
    };
        \draw[blue] (m-1-2.north east) node[scale=.7,above,yshift=-1.5mm] {$-1$};
        \draw[darkgreen] (m-2-3.north east) node[scale=.7,above,yshift=-2mm,xshift=1mm] {$-1$};
    \draw[<->,thick,red] (m-1-3.south west)--(m-3-1.north east);
    \draw[<->,thick,blue] (m-1-2.south west)--(m-2-1.north east);   
    \draw[<->,thick,darkgreen] (m-2-3.south west)--(m-3-2.north east);   
\end{tikzpicture}&
\end{array}$$
\caption{The behavior of $\bijd_1$, $\bijd_2$, $\bija_1$, $\bija_2$, $\rot$ and $\rev$ on the statistics $\b_{r,s}$ when $k=3$. Black squares indicate preserved entries, and colored boxes with rounded corners indicate swapped entries.}
\label{fig:k=3}
\end{figure}
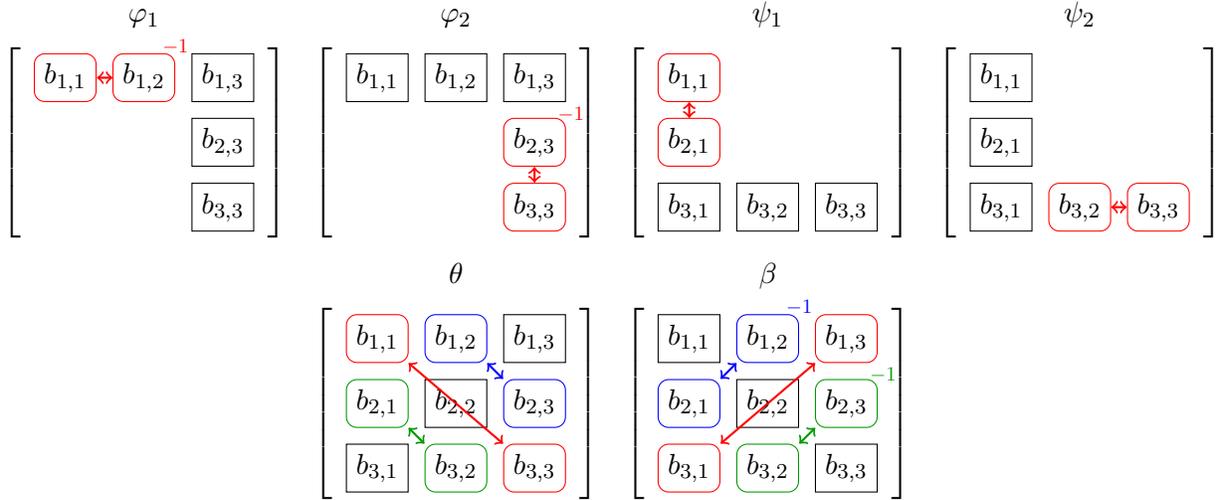

\begin{proof}
Parts (a) and (b) follow from Lemma~\ref{lem:brs} and Remark~\ref{rem:bijd1}. Parts (c) and (d) follow from Lemma~\ref{lem:brs-bija} and the analogous remark. 
Part (e) follows from Lemma~\ref{lem:rot}.

For part (f), Lemma~\ref{lem:revr} implies that $\rev$ swaps  $\b_{2,1}\swap\b_{1,2}-1$ and $\b_{2,3}-1\swap\b_{3,2}$, and that it preserves $\b_{1,1}$, $\b_{2,2}$ and $\b_{3,3}$. Since $\b_{1,3}=n-\b_{1,1}-\b_{1,2}$ and $\b_{3,1}=n-1-\b_{1,1}-\b_{2,1}$, it must also swap $\b_{1,3}\swap\b_{3,1}$.
\end{proof}

In \cite[Sec.~3]{sulanke_three_2005}, Sulanke studies some statistics on $\SYT(n^3)$ that can be described as linear combinations of the $\b_{r,s}$. Using certain operations that produce equidistributed statistics, he shows that some of them (columns 1 and 2 of \cite[Tab.~2]{sulanke_three_2005}\footnote{In Sulanke's notation, the indexing of the matrices is rotated by $180$ degrees.}) have a $3$-Narayana distribution, and he conjectures that the same is true for the others.
Sulanke's operations can be used to reduce the statistics in these conjectures to two cases, corresponding to each of columns 3 and 4 of  \cite[Tab.~2]{sulanke_three_2005}. They can be expressed, respectively, as follows.

\begin{conjecture}[{\cite[Conj.~1]{sulanke_three_2005}}]\label{conj:sulanke}
On $\SYT(n^3)$, each of the statistics $\st_1=\b_{1,2}+\b_{2,2}+\b_{2,3}-2$ and $\st_2=\b_{1,1}+\b_{1,3}+\b_{2,3}-1$ has a $3$-Narayana distribution.
\end{conjecture}

We can prove the second part of the conjecture using the biijection $\bijd_1$. 

\begin{proposition}
On $\SYT(n^3)$, the statistic $\st_2$ has a $3$-Narayana distribution.
\end{proposition}

\begin{proof}
Lemma~\ref{lem:3preserved}(a) implies that, for any $T\in\SYT(n^3)$,
$$\des(T)=\b_{1,2}(T)+\b_{1,3}(T)+\b_{2,3}(T)=\b_{1,1}(\bijd_1(T))+\b_{1,3}(\bijd_1(T))+\b_{2,3}(\bijd_1(T))+1=\st_2(\bijd_1(T))+2.$$ It follows that  $\st_2$ is equidistributed with $\mathrm{des}-2$, which has a $3$-Narayana distribution by equation~\eqref{eq:Ndes}.
\end{proof}

The statement about $\st_1$ remains a conjecture. In fact, computational data on the joint distribution of the statistics $\b_{r,s}$ on $\SYT(n^3)$ for small values of $n$ suggests that much more is true. 

\begin{conjecture}\label{conj:alpha}
There exist involutions $\alpha,\delta:\SYT(n^3)\to\SYT(n^3)$ with the following properties:
\begin{itemize}
\item $\alpha$ swaps $\b_{1,1}\swap\b_{2,1}$, $\b_{1,2}\swap\b_{2,3}$, $\b_{1,3}\swap\b_{2,2}$ and $\b_{3,2}\swap\b_{3,3}$, and it preserves $\b_{3,1}$;
\item $\delta$ swaps $\b_{2,1}\swap\b_{2,3}-1$, and it preserves $\b_{1,2}$, $\b_{2,2}$ and $\b_{3,2}$.
\end{itemize}
\end{conjecture}

The existence of either $\alpha$ or $\delta$ would imply the statement about $\st_1$ in Conjecture~\ref{conj:sulanke}. For example, if $\alpha$ exists, then
$$\des(T)=\b_{1,2}(T)+\b_{1,3}(T)+\b_{2,3}(T)=\b_{2,3}(\alpha(T))+\b_{2,2}(\alpha(T))+\b_{1,2}(\alpha(T))=\st_1(\alpha(T))+2.$$

Given the richness of the joint distribution of the statistics $\b_{r,s}$ already for $k=3$, we expect the study of these statistics for general $k$ to be quite interesting.

\subsection{A family of $k$-Narayana statistics, and a connection to canon permutations}

A family of statistics on $\SYT(n^k)$ interpolating between ascents and descents was defined by Sulanke in~\cite{sulanke_generalizing_2004}, and recently studied in~\cite{elizalde_canon_2025}. These statistics are indexed by the set of permutations of $\{1,2,\dots,k\}$, which we denote by $\S_k$. For each $\sigma\in\S_k$, let
$$\des_\sigma=\sum_{\substack{r,s \\ \sigma_r>\sigma_s}} \b_{r,s}=\sum_{1\le i<j\le k} \b_{\sigma^{-1}_j,\sigma^{-1}_i}.$$
These statistics generalize ascents and descents in the sense that, for $T\in\SYT(n^k)$,
$$\des_{12\dots k}(T)=\sum_{r>s} \b_{r,s}(T)=\asc(T) \quad\text{and}\quad \des_{k\dots21}(T)=\sum_{r<s} \b_{r,s}(T)=\des(T).$$

It was conjectured by Sulanke (see the remark at the end of~\cite[Sec.~3.1]{sulanke_generalizing_2004}), and later proved
by the author (see~\cite[Eq.~(7)]{elizalde_canon_2025}), that, with an appropriate shift, all these statistics have the same distribution.
In the next theorem, $\des(\sigma)=|\{r:\sigma_r>\sigma_{r+1}\}|$ is the number of descents of the permutation~$\sigma$.

\begin{theorem}[\cite{elizalde_canon_2025}\footnote{In the notation used in~\cite{elizalde_canon_2025}, the roles of $k$ and $n$ are switched.}]
\label{thm:sigma}
For every $\sigma\in\S_k$, there exists a bijection $\phi_\sigma:\SYT(n^k)\to\SYT(n^k)$ such that, for all $T\in\SYT(n^k)$,
$$\des_\sigma(T)=\asc(\phi_\sigma(T))+\des(\sigma).$$ 
In particular, the statistic $\des_\sigma-\des(\sigma)$ on $\SYT(n^k)$ has a $k$-Narayana distribution.
\end{theorem}

The bijection $\rev$ from Theorem~\ref{thm:rev} proves Theorem~\ref{thm:sigma} in the special case that $\sigma=k\dots21$.
However, unlike our description of $\rev$, the map $\phi_\sigma$ provided in~\cite{elizalde_canon_2025} is quite complicated. It is described as a composition of multiple bijections, and it is not an involution in general.

The original motivation for Theorem~\ref{thm:sigma} comes from the study of descents on canon permutations. These are permutations of the multiset consisting of $n$ copies of each number in $\{1,2,\dots,k\}$, with the property that
the subsequences obtained by taking the $j$th copy of each entry are identical for any given $j$. Denote the set of such permutations by $\C^{n}_k$. For example, $313321214424$ is a canon permutation in $\C^{3}_4$ because the three subsequences are equal to $3124$. As shown in~\cite{elizalde_canon_2025}, there is a straightforward bijection  
\begin{equation}\label{eq:canon}
\begin{array}{ccc}
\S_k\times\SYT(n^k)&\longrightarrow&\C^{n}_k\\
(\sigma,T)&\mapsto&\pi=\sigma_{\row(1)}\sigma_{\row(2)}\dots \sigma_{\row(kn)},
\end{array}
\end{equation}
which satisfies $\des_\sigma(T)=\des(\pi)$.

Canon permutations were introduced in~\cite{elizalde_descents_2024} as a variation of {\em quasi-Stirling permutations}~\cite{archer_pattern_2019,elizalde_descents_2021} and {\em Stirling permutations}~\cite{gessel_stirling_1978}.
When $k=2$, canon permutations are sometimes called nonnesting permutations, as they are in bijection with labeled nonnesting matchings. 
 In~\cite{elizalde_canon_2025}, generalizing the case $k=2$ proved in~\cite{elizalde_descents_2024}, it was shown that the polynomial that enumerates permutations in $\C^{n}_k$ by the number of descents has a nice factorization as a product of an Eulerian polynomial and a Narayana polynomial $\sum_{h} N(k,n,h)\, t^h$. A different proof using the theory of $(P,w)$-partitions has recently been given by Beck and Deligeorgaki~\cite{beck_canon_2024}.

This factorization, together with the palindromicity of the Eulerian and Narayana polynomials, implies that the distribution of the number of descents on~$\C^{n}_k$ is symmetric. However, no bijective proof of this fact was known. Using our bijection $\bija$ from Theorem~\ref{thm:bija}, together with the bijection $\phi_\sigma$ from Theorem~\ref{thm:sigma}, we can now provide a bijective proof of this symmetry. For $\sigma\in\S_k$, denote by $\sigma^c\in\S_k$ the permutation such that $\sigma^c_r=k+1-\sigma_r$ for all~$r$.

We define a bijection $\pi\mapsto\wh{\pi}$ on $\C^{n}_k$ as follows. Given $\pi\in \C^{n}_k$ corresponding to the pair $(\sigma,T)$ under the map in equation~\eqref{eq:canon}, let $\wh{\pi}$ be the canon permutation corresponding to the pair $(\sigma^c,\wh{T})$, where $\wh{T}=\phi_{\sigma^c}^{-1}(\bija(\phi_\sigma(T)))$.

\begin{proposition}
The above map $\pi\mapsto\wh{\pi}$ is an involution on $\C^{n}_k$ that satisfies $$\des(\pi)+\des(\wh{\pi})=(k-1)n.$$
\end{proposition}

\begin{proof}
It is clear that the map is an involution, since $(\sigma^c)^c=\sigma$ and $\phi_{\sigma}^{-1}(\bija(\phi_{\sigma^c}(\wh{T})))=T$, using that $\bija$ is an involution by Lemma~\ref{lem:involution}.
Additionally, 
\begin{align*}\des(\wh\pi)&=\des_{\sigma^c}(\wh{T})=\des_{\sigma^c}(\phi_{\sigma^c}^{-1}(\bija(\phi_\sigma(T))))  && \text{(by equation~\eqref{eq:canon})}\\
&=\asc(\bija(\phi_\sigma(T)))+\des(\sigma^c) && \text{(by Theorem~\ref{thm:sigma})}\\
&=(k-1)(n-1)-\asc(\phi_\sigma(T))+\des(\sigma^c) && \text{(by Theorem~\ref{thm:bija})}\\
&=(k-1)(n-1)-\des_\sigma(T)+\des(\sigma)+\des(\sigma^c)&& \text{(by Theorem~\ref{thm:sigma})}\\
&=(k-1)n-\des_\sigma(T) && \text{(since $\des(\sigma)+\des(\sigma^c)=k-1$)} \\
&=(k-1)n-\des(\pi). && \text{(by equation~\eqref{eq:canon})} \qedhere
\end{align*}
\end{proof}

\section{Descent polynomials of graded posets}\label{sec:posets}
In this section we put our results in the bigger context of linear extensions of posets.
We refer the reader to~\cite[Sec.~3.15]{stanley_enumerative_2012} for more details.
Let $P$ be a poset (partially ordered set) with $p$ elements. A chain in $P$ is a set of elements satisfying $x_1<\dots<x_m$. We say that $P$ is {\em graded} if every maximal chain has the same number of elements.

A {\em linear extension} of $P$ is an order-preserving bijection $\sigma:P\to\{1,2,\dots,p\}$. 
Denote by $\cL(P)$ the set of linear extensions of $P$. 
Fix a particular linear extension $\omega$, called a {\em natural labeling}, and identify the elements of $P$ with their labels. Then, every $\sigma\in\cL(P)$ can be identified with a permutation of the labels, namely $\omega(\sigma^{-1}(1)),\dots,\omega(\sigma^{-1}(p))$. We say that $i$ is a descent of $\sigma$ if $\omega(\sigma^{-1}(i))>\omega(\sigma^{-1}(i+1))$.
The {\em descent polynomial} of $P$ enumerates its linear extensions by the number of descents. It is known that it does not depend on the natural labeling $\omega$.

\subsection{Palindromicity of the descent polynomial}

The following result follows from \cite[Prop.~19.3]{stanley_ordered_1972}, see also \cite[Cor.~3.15.18]{stanley_enumerative_2012} and~\cite{farley_linear_2005}.

\begin{theorem}[Stanley, 1972]\label{thm:stanley}
For any graded poset $P$ with $p$ elements where every maximal chain has $m$ elements, the number of linear extensions of $P$ with $h$ descents equals the number of those with $p-m-h$ descents, for all $0\le h\le p-m$.
\end{theorem}

An equivalent way to state the above symmetry is by saying that the descent polynomial of any graded poset is palindromic. In fact, Stanley proved \cite[Prop.~19.3]{stanley_ordered_1972} that being graded is not only sufficient but also necessary for the descent polynomial to be palindromic.

Theorem~\ref{thm:symmetry} is a special case of Theorem~\ref{thm:stanley} because one can view elements of $\SYT(n^k)$ as linear extensions of the graded poset $\mathbf{k}\times\mathbf{n}$, 
which is the set $\{(r,c):1\le r\le k, \, 1\le c\le n\}$ with order relation $(r,c)\le(r',c')$ if and only if $r\le r'$ and $c\le c'$. Specifically, for any $T\in\SYT(n^k)$, define a linear extension $\sigma$ by letting $\sigma(r,c)=i$ if $i$ appears in row $r$ and column $c$ of $T$.
With the natural labeling of $\mathbf{k}\times\mathbf{n}$ given by $\omrow(r,c)=c+(r-1)n$ (see the left of Figure~\ref{fig:kxn}), ascents of $T$ correspond to descents of the associated linear extension $\sigma$. 
Since this poset has $kn$ elements and its maximal chains have $k+n-1$ elements, Theorem~\ref{thm:symmetry} is equivalent to Theorem~\ref{thm:stanley} in the case $P=\mathbf{k}\times\mathbf{n}$.

\begin{figure}[htbp]
    \centering
\begin{tikzpicture}[scale=.5]
\foreach \x in {0,...,3}
	{\foreach \y in {0,...,4}
		{\fill (\x - \y, \x + \y) circle (0.1cm) {} node[left]{\pgfmathparse{\x*5+\y+1}\pgfmathprintnumber{\pgfmathresult}};
	    \ifthenelse{\x < 3}
			{\draw (\x - \y, \x + \y) -- (\x - \y + 1, \x + \y + 1);}{}
		\ifthenelse{\y < 4}
			{\draw (\x - \y, \x + \y) -- (\x - \y - 1, \x + \y+1);}{}
		}
	}
\end{tikzpicture}
\qquad\qquad
\begin{tikzpicture}[scale=.5]
\foreach \x in {0,...,2}
	{\foreach \y in {0,...,2}
		{\fill (\x - \y, \x + \y) circle (0.1cm) {} node[left]{\pgfmathparse{\x*(11-\x)/2+\y+1}\pgfmathprintnumber{\pgfmathresult}};
	    \ifthenelse{\x < 2}
			{\draw (\x - \y, \x + \y) -- (\x - \y + 1, \x + \y + 1);}{}
		\ifthenelse{\y < 2}
			{\draw (\x - \y, \x + \y) -- (\x - \y - 1, \x + \y+1);}{}
		}
	}
\foreach \z in {(-3,3),(-4,4),(-2,4)} 
	{\fill \z circle (0.1cm) {}; \draw \z -- ++ (1,-1); }
\draw (-4,4) node[left]{5}; \draw (-3,3) node[left]{4}; \draw (-2,4) node[left]{9};
\draw (-3,3)--(-2,4);
\end{tikzpicture}
\caption{The poset $\mathbf{k}\times\mathbf{n}$ for $k=4$ and $n=5$, corresponding to $\lambda=(5^4)$ (left), and the poset corresponding to $\lambda=(5,4,3)$ (right), both with natural labelings $\omrow$ that order the cells lexicographically by rows.}
\label{fig:kxn}
\end{figure}
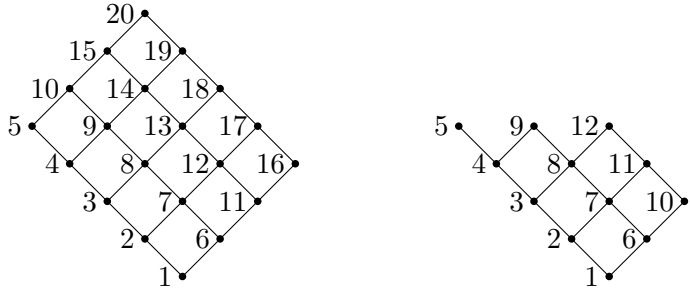

However, no purely bijective proofs of Theorems~\ref{thm:symmetry} and~\ref{thm:stanley} appear in the literature. Stanley's original proof relies on the theory of $P$-partitions and order polynomials, as does Sulanke's proof of Theorem~\ref{thm:symmetry}. 
According to Farley~\cite{farley_linear_2005}, Stanley posed the problem of finding a combinatorial proof of Theorem~\ref{thm:stanley} back in 1981, and specifically mentioned the poset $\mathbf{k}\times\mathbf{n}$ as an interesting case.
In~\cite{gasharov_neggers-stanley_1998},  Gasharov gave a combinatorial proof for graded posets whose maximal chains have at most $3$ elements.
In~\cite{farley_linear_2005}, Farley proved Theorem~\ref{thm:stanley} by giving a general bijective contruction that relies on the involution principle of Garsia and Milne, and left the open problem of finding a more natural bijection for particular posets, and specifically for $\mathbf{k}\times\mathbf{n}$.

Each of our Theorems~\ref{thm:bijd} and~\ref{thm:bija} provides a direct (involution-principle-free) bijective proof of Theorem~\ref{thm:symmetry}, where the generalized Narayana numbers are interpreted as counting descents and ascents, respectively.
As shown in Section~\ref{sec:LK}, in the special case of two-row tableaux, both bijections $\bijd$ and $\bija$ are equivalent to the Lalanne--Kreweras involution on Dyck paths. 
Generalizations of this classical involution in a different direction have been considered by Hopkins and Joseph~\cite{hopkins_birational_2022}, who extend it to the piecewise-linear and birational realms, and describe a more general involution, called {\em rowvacuation}, on the set of order ideals of any graded poset. However, these generalizations do not help when dealing with rectangular tableaux with more than two rows.

For any partition $\lambda$, one can define the poset $P_\lambda$ whose elements are $\{(r,c):1\le r\le k, 1\le c\le\lambda_r\}$ ordered coordinate-wise, so that tableaux in $\SYT(\lambda)$ can be viewed as linear extensions of $P_\lambda$. Again, with the natural labeling $\omrow$ of $P_\lambda$ that orders cells lexicographically by rows, ascents of a tableau in $\SYT(\lambda)$ correspond to descents of associated linear extension. With this setup, Theorem~\ref{thm:staircase} provides a direct bijection proving Theorem~\ref{thm:stanley} in the special case of posets $P_\lambda$ where $\lambda=(n,n-1,\dots,n-k+1)$;
see the right of Figure~\ref{fig:kxn} for an example. Note that $P_\lambda$ is a graded poset with $\frac{(2n-k+1)k}{2}$ elements, whose maximal chains have $n$ elements.

\subsection{Independence on natural labeling}\label{sec:independence}

The symmetry of the generalized Narayana numbers stated in Proposition~\ref{prop:n<->k} is a special case of the fact that the descent polynomial of a poset does not depend on the underlying natural labeling. This independence follows from~\cite[Thm.~3.15.8]{stanley_enumerative_2012}, noting that  when $\omega$ is natural, $(P,\omega)$-partitions are simply order-reversing maps from $P$ to the natural numbers.
To understand the connection to generalized Narayana numbers, first recall that descents of a linear extension of $\mathbf{k}\times\mathbf{n}$ with respect to the natural labeling $\omrow(r,c)=c+(r-1)n$ (lexicographically by rows) correspond to ascents of the associated tableau in $\SYT(n^k)$ (c.f.\ the left-hand side in Proposition~\ref{prop:n<->k}).
On the other hand, descents of a linear extension of  $\mathbf{k}\times\mathbf{n}$ with respect to the natural labeling $\omcol(r,c)=r+(c-1)k$ (lexicographically by columns) correspond to high descents of the associated tableau in $\SYT(n^k)$, which in turn correspond to ascents in the conjugate tableau in $\SYT(k^n)$  (c.f.\ the right-hand side in Proposition~\ref{prop:n<->k}).
A similar argument shows that the statistics $\asc$ and $\hdes$ are equidistributed over $\SYT(\lambda)$ for any given shape $\lambda$.

The approach from \cite[Sec.~3.15]{stanley_enumerative_2012} does not provide a bijective proof. However, inspired on similar ideas, we can give the following inclusion-exclusion argument to show that the distribution of the {\em descent set}---not only the number of descents---over linear extensions does not depend on the natural labeling.
Given a natural labeling $\omega$ of $P$ and a linear extension $\sigma$, denote by $\Des_\omega(\sigma)$ the set of descents of $\sigma$ with respect to $\omega$.

\begin{proposition}\label{prop:inclusion-exclusion}
For any poset $P$ and any two natural labelings $\omega$ and $\omega'$, 
there is a recursively-defined bijection $f:\cL(P)\to\cL(P)$ such that, for all $\sigma\in\cL(P)$, we have
$\Des_\omega(\sigma)=\Des_{\omega'}(f(\sigma))$.
\end{proposition}

\begin{proof}
Let $P$ be a poset with $p$ elements, and fix $S\subseteq \{1,2,\dots,p-1\}$. 
Denote the elements of $S$ by $s_1<s_2<\dots<s_m$, and define $s_0=0$ and $s_{m+1}=p$. 

Let us first describe bijections
\begin{equation}\label{eq:f}
f_{\subseteq S}:\{\sigma\in\cL(P):\Des_\omega(\sigma)\subseteq S\}\to\{\tau\in\cL(P):\Des_{\omega'}(\tau)\subseteq S\}.
\end{equation}
Given a linear extension $\sigma$ on the left-hand side, let 
$$P_j=\{x\in P:s_j<\sigma(x)\le s_{j+1}\}$$
for each $0\le j\le m$. Since $\sigma$ does not have any descents with respect to $\omega$ strictly between $s_j$ and $s_{j+1}$, the values of $\sigma$ on $P_j$ increase in the order given by $\omega$, that is,
\begin{equation}\label{eq:omega}\omega(\sigma^{-1}(s_j))<\omega(\sigma^{-1}(s_j+1))<\dots<\omega(\sigma^{-1}(s_{j+1})).
\end{equation}
Define $f_{\subseteq S}(\sigma)=\tau$ to be the unique linear extension such that, for each $j$, we have $s_j<\tau(x)\le s_{j+1}$ for all $x\in P_j$, and the values of $\tau$ on $P_j$ increase in the order given by $\omega'$, that is, 
$$\omega'(\tau^{-1}(s_j))<\omega'(\tau^{-1}(s_j+1))<\dots<\omega'(\tau^{-1}(s_{j+1})).$$
By construction, $\tau$ does not have any descents with respect to $\omega'$ strictly between $s_j$ and $s_{j+1}$, so $\Des_{\omega'}(\tau)\subseteq S$. Additionally, the map $f_{\subseteq S}$ is invertible because, given $\tau$, the linear extension $\sigma$ is uniquely determined by the sets $P_j=\{x\in P:s_j<\tau(x)\le s_{j+1}\}$ and the condition~\eqref{eq:omega}.

These bijections prove that, for any $S\subseteq \{1,2,\dots,p-1\}$, the two sets in equation~\eqref{eq:f} have the same cardinality. From this fact, the principle of inclusion-exclusion implies that 
$$\card{\{\sigma\in\cL(P):\Des_\omega(\sigma)= S\}}=\card{\{\tau\in\cL(P):\Des_{\omega'}(\tau)= S\}}$$
for all $S$.

It is possible to recursively describe bijections
$$f_{=S}:\{\sigma\in\cL(P):\Des_\omega(\sigma)= S\}\to\{\tau\in\cL(P):\Des_{\omega'}(\tau)= S\}$$
as follows, using ideas from~\cite{ferreri_generating_2024}.

When $S=\emptyset$, the map $f_{=\emptyset}=f_{\subseteq\emptyset}:\{\omega\}\mapsto\{\omega'\}$ is trivial since the two sets have cardinality one.
Suppose now that $S\neq\emptyset$, and that $f_{=R}$ has been defined for all proper subsets $R\subset S$. Let $\sigma\in\cL(P)$ with $\Des_\omega(\sigma)= S$. To compute $f_{=S}(\sigma)$, we start by letting $\tau=f_{\subseteq S}(\sigma)$ and $R=\Des_{\omega'}(\tau)\subseteq S$. 
As long as $R\neq S$, we let $\tau\coloneqq f_{\subseteq S}( f_{=R}^{-1}(\tau))$, let $R\coloneqq\Des_{\omega'}(\tau)$, and repeat this step. When eventually $R=S$, we define $f_{=S}(\sigma)=\tau$. It is easy to check that this process always terminates, and that $f_{=S}$ is invertible.

Finally, to define $f$ for an arbitrary $\sigma\in\cL(P)$, let $f(\sigma)=f_{=S}(\sigma)$ where $S=\Des_\omega(\sigma)$.
\end{proof}

\begin{corollary}\label{cor:Asc-HDes}
For any partition $\lambda$, 
there is a recursively-defined bijection $f:\SYT(\lambda)\to\SYT(\lambda)$ such that, for all $T\in\SYT(\lambda)$, we have
$\Asc(T)=\HDes(f(T))$.
\end{corollary}

\begin{proof}
Let $P=P_\lambda$ in Proposition~\ref{prop:inclusion-exclusion}, and let $\omega=\omrow$ and $\omega'=\omcol$ be the natural labelings of $P_\lambda$ that order the cells lexicographically by rows and by columns, respectively. If $\sigma$ is a linear extension of $P_\lambda$ and $T\in\SYT(\lambda)$ is its corresponding tableau, we have $\Des_\omega(\sigma)=\Asc(T)$ and $\Des_{\omega'}(\sigma)=\HDes(T)$. In particular, $f_{=S}$ is a bijection between $\{T\in\SYT(\lambda):\Asc(T)= S\}$ and $\{T\in\SYT(\lambda):\HDes(T)= S\}$.
\end{proof}

See Figure~\ref{fig:Asc-HDes} for an example of this bijection.
\newcommand\nine{\textcolor{red}{9}}
\newcommand\seven{\textcolor{blue}{7}}
\newcommand\eight{\textcolor{blue}{8}}

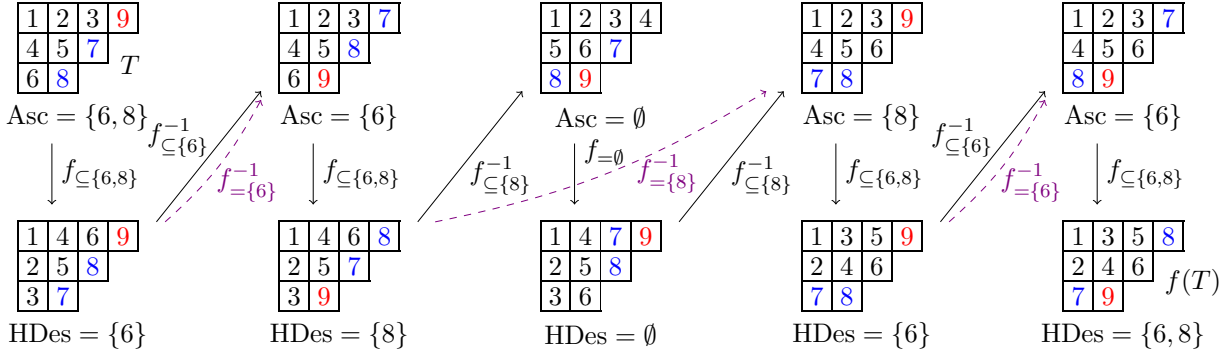
\begin{figure}[htb]
\resizebox{\textwidth}{!}{
\begin{tikzpicture}[scale=1.25]
\node at (.6,-.2) {$T$};
\node at (0,0) {$\young(123\nine,45\seven,6\eight)$};
\node at (0,-.8) {$\Asc=\{6,8\}$};
\draw[->] (-.3,-1.1)-- node[right]{$f_{\subseteq\{6,8\}}$} (-.3,-1.8);
\begin{scope}[shift={(0,-2.5)}]
\node at (0,0) {$\young(146\nine,25\eight,3\seven)$};
\node at (0,-.8) {$\HDes=\{6\}$};
\draw[->] (.9,0.5)--  node[above left=-4]{$f^{-1}_{\subseteq\{6\}}$}  (2.1,2);
\draw[violet,dashed,->,bend right=10] (1,0.5) to  node[below right=-5]{$f^{-1}_{=\{6\}}$}  (2.1,1.9);
\end{scope}
\begin{scope}[shift={(3,0)}]
\node at (0,0) {$\young(123\seven,45\eight,6\nine)$};
\node at (0,-.8) {$\Asc=\{6\}$};
\draw[->] (-.3,-1.1)-- node[right]{$f_{\subseteq\{6,8\}}$} (-.3,-1.8);
\begin{scope}[shift={(0,-2.5)}]
\node at (0,0) {$\young(146\eight,25\seven,3\nine)$};
\node at (0,-.8) {$\HDes=\{8\}$};
\draw[->] (.9,0.5)--  node[below right=-4]{$f^{-1}_{\subseteq\{8\}}$}  (2.1,2);
\draw[violet,dashed,->,bend right=10] (1.1,0.5) to node[right=8]{$f^{-1}_{=\{8\}}$}  (4.9,2);
\end{scope}
\end{scope}
\begin{scope}[shift={(6,0)}]
\node at (0,0) {$\young(1234,56\seven,\eight\nine)$};
\node at (0,-.8) {$\Asc=\emptyset$};
\draw[->] (-.3,-1.1)-- node[above right]{$f_{=\emptyset}$} (-.3,-1.8);
\begin{scope}[shift={(0,-2.5)}]
\node at (0,0) {$\young(14\seven\nine,25\eight,36)$};
\node at (0,-.8) {$\HDes=\emptyset$};
\draw[->] (.9,0.5)--  node[below right=-4]{$f^{-1}_{\subseteq\{8\}}$}  (2.1,2);
\end{scope}
\end{scope}
\begin{scope}[shift={(9,0)}]
\node at (0,0) {$\young(123\nine,456,\seven\eight)$};
\node at (0,-.8) {$\Asc=\{8\}$};
\draw[->] (-.3,-1.1)-- node[right]{$f_{\subseteq\{6,8\}}$} (-.3,-1.8);
\begin{scope}[shift={(0,-2.5)}]
\node at (0,0) {$\young(135\nine,246,\seven\eight)$};
\node at (0,-.8) {$\HDes=\{6\}$};
\draw[->] (.9,0.5)--  node[above left=-4]{$f^{-1}_{\subseteq\{6\}}$}  (2.1,2);
\draw[violet,dashed,->,bend right=10] (1,0.5) to  node[below right=-5]{$f^{-1}_{=\{6\}}$}  (2.1,1.9);
\end{scope}
\end{scope}
\begin{scope}[shift={(12,0)}]
\node at (0,0) {$\young(123\seven,456,\eight\nine)$};
\node at (0,-.8) {$\Asc=\{6\}$};
\draw[->] (-.3,-1.1)-- node[right]{$f_{\subseteq\{6,8\}}$} (-.3,-1.8);
\begin{scope}[shift={(0,-2.5)}]
\node at (0,0) {$\young(135\eight,246,\seven\nine)$};
\node at (0,-.8) {$\HDes=\{6,8\}$};
\node at (.8,-.2) {$f(T)$};
\end{scope}
\end{scope}
\end{tikzpicture}}
\caption{An example of the bijection $f$ from Corollary~\ref{cor:Asc-HDes} on $\SYT(4,3,2)$.}
\label{fig:Asc-HDes}
\end{figure}

\section{Further directions}\label{sec:open}

\subsection{Symmetry of the major index}

Another closely related statistic on standard Young tableaux is the {\em major index}, defined as $\maj(T)=\sum_{i\in\Des(T)} i$. For any partition $\lambda$ of $N$, the distribution of this statistic over standard Young tableaux of shape $\lambda$ is given by Stanley's \cite[Cor.~7.21.5]{stanley_enumerative_1999} $q$-analogue of the hook length formula:
\begin{equation}\label{eq:HLF}\sum_{T\in\SYT(\lambda)} q^{\maj(T)}=q^{b(\lambda)}\frac{[N]_q!}{\prod_{c\in\lambda}[h_c]_q!},
\end{equation}
where $h_c$ is the hook length of a cell $c$ in $\lambda$, $b(\lambda)=\sum_j (j-1)\lambda_j$, and $[a]_q!=\prod_{j=1}^a(1+q+\dots+q^{j-1})$.
The coefficients of these polynomials are important in algebraic combinatorics, 
and questions about their unimodality, asymptotic behavior, and internal zeros have been studied in~\cite{adin_descent_2001,billey_asymptotic_2020}.
It follows from equation~\eqref{eq:HLF} that these polynomials are palindromic. However, no bijective proof of this fact seems to be known. 
\begin{problem}\label{prob:maj}
For any partition $\lambda$ of $N$, give an explicit bijection $\Phi:\SYT(\lambda)\to\SYT(\lambda)$ such that, for all $T\in\SYT(\lambda)$,
$$\maj(T) + \maj(\Phi(T)) = \binom{N}{2} + b(\lambda) - b(\lambda').$$
\end{problem}

Unfortunately, our bijection $\bijd$ for the case $\lambda=(n^k)$ (where $\binom{N}{2} + b(\lambda) - b(\lambda')=\frac{k(k-1)n(n+1)}{2}$)
does not have this property, unless $k=2$ or $n=2$. We note that if we restrict to tableaux in $\SYT(\lambda)$ with a fixed number $d$ of descents, then Sch\"utzenberger's evacuation map shows that the distribution of $\maj$ on this subset is symmetric, since $\des(T)=\des(\evac(T))$ and $\maj(T)+\maj(\evac(T))=dN$. However, this does not solve Problem~\ref{prob:maj}.

In the case of rectangular shapes $\lambda=(n^k)$, computational evidence for $k+n\le 9$ suggests that, in fact, the distributions of the number of descents and the major index are ``jointly symmetric''. We wonder if some modification of our bijection $\bijd$ could be used to prove this.

\begin{problem}\label{prob:desmaj}
Describe a bijection $\Phi :\SYT(n^k)\to\SYT(n^k)$ such that, for all $T\in\SYT(n^k)$,
$$\des(T) + \des(\Phi(T)) = (k-1)(n+1) \quad\text{and}\quad \maj(T) + \maj(\Phi(T)) = \frac{k(k-1)n(n+1)}{2}.$$
\end{problem}

\subsection{A rowmotion map on standard Young tableaux}

Promotion and rowmotion are maps defined on the set of order ideals of a poset. They play an important role in the emerging field of dynamical algebraic combinatorics~\cite{striker_promotion_2012}.
In the special case of the type $A$ root poset $\mathbf{A}_{n-1}$, its order ideals are in natural correspondence with the set $\D_n$ of Dyck paths, which in turn are in bijection with $\SYT(n^2)$ via the map $\TD$ from equation~\eqref{eq:TD}. 
Under these bijections, promotion of order ideals of $\mathbf{A}_{n-1}$ translates to promotion on $\SYT(n^2)$, an extensively studied map defined not only on standard Young tableaux of any shape, but on linear extensions of any poset~\cite{stanley_promotion_2009}. On the other hand, finding a natural rowmotion operation on standard Young tableaux remains an open question in dynamical algebraic combinatorics.

In the special case of $\SYT(n^2)$, one can define a rowmotion operation by simply translating rowmotion on order ideals of $\mathbf{A}_{n-1}$ via the above bijections. Denoting this map by $\usualrowmotion_2$, it has the property that $\Asc(T)=\HDes(\usualrowmotion_2(T))$ for all $T\in\SYT(n^2)$.
By comparison, the recursive bijection $f$ from Corollary~\ref{cor:Asc-HDes} is defined on $\SYT(\lambda)$ for any $\lambda$, and it also satisfies 
$\Asc(T)=\HDes(f(T))$ for all $T\in\SYT(\lambda)$. However, the bijection $f$ has two major drawbacks: it does not coincide with $\usualrowmotion_2$ on $\SYT(n^2)$, and its definition is recursive.

Instead, Proposition~\ref{prop:rowmotion} suggests a more natural candidate for a rowmotion operation on standard Young tableaux. 
It can be shown that this bijection $\rowmotion:\SYT(n^k)\to\SYT(n^k)$ 
coincides with $\usualrowmotion_2$ when $k=2$.
And while it does not satisfy $\Asc(T)=\HDes(\rowmotion(T))$ in general, it turns out that it preserves a different refinement of the number of ascents and high descents. Specifically, denote by $(r_i,c_i)$ the cell in $T$ that contains $i$, and let
\begin{align*} 
\AscCell(T)&=\{\{(r_i,c_i),(r_{i+1},c_{i+1})\}:i\in\Asc(T)\},\\
\HDesCell(T)&=\{\{(r_i,c_i),(r_{i+1},c_{i+1})\}:i\in\HDes(T)\}.
\end{align*}
These statistics record the pairs of cells in $T$ that contain the entries $i$ and $i+1$ when $i$ is an ascent or a high descent, respectively. By definition, $\card{\AscCell(T)}=\asc(T)$ and  $\card{\HDesCell(T)}=\hdes(T)$.
In follow-up work, we will show that there is an alternative, more direct description of $\rowmotion$ that allows us to extend it to $\SYT(\lambda)$ for any $\lambda$, and to prove that $\AscCell(T)=\HDesCell(\rowmotion(T))$ for all $T\in\SYT(\lambda)$.

\subsection{Generalizations to other posets}

As discussed in Section~\ref{sec:posets}, Theorems~\ref{thm:bija} and~\ref{thm:staircase} provide direct bijections proving Theorem~\ref{thm:stanley} in special cases. It is natural to ask whether our involution $\bija$ can be generalized to other shapes $\lambda$ for which the poset $P_\lambda$ is graded, in order to give a direct bijective proof of the symmetry of the distribution of $\asc$ on $\SYT(\lambda)$. It would also be interesting to find a generalization of $\bija$ to other graded posets, such as the product of three chains $\mathbf{k}\times\mathbf{n}\times\mathbf{m}$.

It follows from \cite[Prop.~19.3]{stanley_ordered_1972} that the partitions $\lambda$ for which the distribution of $\asc$ on $\SYT(\lambda)$ is symmetric are precisely those for which the poset $P_\lambda$ is graded. One could ask if there is an analogous characterization of the partitions $\lambda$ for which the distribution of $\des$ on $\SYT(\lambda)$ is symmetric. We know this is the case for rectangles $\lambda=(n^k)$ (by Theorem~\ref{thm:bijd}), for self-conjugate shapes $\lambda=\lambda'$ (as discussed in Section~\ref{sec:basic}), and for hooks $\lambda=(n,1^{k-1})$ (since all standard Young tableaux of this shape have exactly $k$ descents). However, there are other shapes for which the distribution of $\des$ on $\SYT(\lambda)$ is symmetric as well, the smallest ones being $\lambda=(4,3,1)$ and its conjugate.

\subsection*{Acknowledgements}
The author thanks Ron Adin, Matthias Beck, Danai Deligeorgaki, Jonathan Farley, Sam Hopkins, Yuval Roichman, Richard Stanley, and Jessica Striker for interesting discussions. This work was partially supported by Simons Collaboration Grant \#929653.

\bibliographystyle{plain}
\bibliography{symmetry_des_rectangular_SYT}

\end{document}